\documentclass[11pt]{article}

\usepackage[utf8]{inputenc}
\usepackage[english] {babel}
\usepackage{easyReview}
\usepackage{mathrsfs}

\usepackage{amsmath}
\usepackage{amsfonts}
\usepackage{amssymb}
\usepackage{amsthm}
\usepackage{mathtools}
\usepackage{bbm,a4wide}

\usepackage{braket}
\usepackage{enumerate}
\usepackage{esint}

\usepackage{color}   %May be necessary if you want to color links
\usepackage{hyperref}
\hypersetup{
    colorlinks=true, %set true if you want colored links
    linktoc=all,     %set to all if you want both sections and subsections linked
    linkcolor=blue,  %choose some color if you want links to stand out
}

\theoremstyle{definition} \newtheorem{deff}{Definition} [section]
\theoremstyle{definition} \newtheorem{prop} [deff]{Proposition}
\theoremstyle{definition} \newtheorem{thm} [deff]{Theorem}
\theoremstyle{definition} \newtheorem{lm} [deff]{Lemma}

\theoremstyle{remark} 
\theoremstyle{remark} \newtheorem{rmk}[deff]{Remark} 
\theoremstyle{plain} \newtheorem{cor}[deff]{Corollary} 

\newtheoremstyle{claim}% Nome dello stile
  {}% Spazio sopra
  {}% Spazio sotto
  {\itshape}% Font del corpo dell'enunciato
  {}% Rientro del corpo dell'enunciato
  {}% Font del titolo dell'enunciato
  {.}% Punteggiatura dopo il titolo dell'enunciato
  { }% Spazio dopo il titolo dell'enunciato
  {}% Specifica dell'enunciato

\theoremstyle{claim}% Imposta lo stile
\mathtoolsset{
	showonlyrefs
}

\newcommand{\concto}{\xrightarrow{conc}}
\newcommand{\slope}[3][\rm{Ent}]{\left|\partial^-#1_{#2}\right|(#3)}
\newcommand{\hgf}[2][t]{\mathcal{H}_{#1}(#2)}

\renewenvironment{rmk}{\begin{oldrmk}}{\hfill$\blacksquare$\end{oldrmk}\\ \par}

\newcommand{\N}{\mathbb{N}}

\newcommand{\R}{\mathbb{R}}
\newcommand{\Z}{\mathbb{Z}}
\newcommand{\aalpha}{{\mbox{\boldmath$\alpha$}}}

\newcommand{\ppi}{{\mbox{\boldmath$\pi$}}}

\newcommand{\saalpha}{{\mbox{\scriptsize\boldmath$\alpha$}}}

\newcommand{\sppi}{{\mbox{\scriptsize\boldmath$\pi$}}}
\newcommand{\restr}[1]{\lower3pt\hbox{$|_{#1}$}}
\newcommand{\Mp}{\mathscr M^+}                  % MISURE DI RADON POSITIVE FINITE SUI LIMITATI
\newcommand{\pr}{\mathscr P}

\newcommand{\X}{{\rm X}}
\newcommand{\Y}{{\rm Y}}
\renewcommand{\Z}{{\rm Z}}

\newcommand{\sfd}{{\sf d}}

\newcommand{\mm}{{\mathfrak m}}
\newcommand{\ric}{{\rm Ric}}								% Curvatura Ricci
\newcommand{\vol}{{\rm vol}}								% Misura di volume
\renewcommand{\L}{{\rm L}} %\renewcommand{\L}{{\sf L}}

\newcommand{\CD}{{\sf CD}}

\newcommand{\RCD}{{\sf RCD}}

\newcommand{\ch}{{\sf Ch}}

\newcommand{\ent}{{\rm Ent}_\mm}
\newcommand{\entm}[1]{{\rm Ent}_{#1}}
\newcommand{\adm}{{\sf Adm}}
\newcommand{\W}{{\rm W}}
\renewcommand{\d}{{\rm d}}
\newcommand{\D}{{\rm D}}		
\renewcommand{\div}{{\rm div}}

\newcommand{\Lip}{{\rm Lip}}

\newcommand{\lims}{\varlimsup}
\newcommand{\limi}{\varliminf}

\renewcommand{\b}{{\rm b}}

\newcommand{\weakto}{\rightharpoonup}
\newcommand{\nchi}{{\raise.3ex\hbox{$\chi$}}}
\newcommand{\supp}{{\rm supp}}
\newcommand{\eps}{\varepsilon}

\newcommand{\B}{\mathcal B}                            % Boreliani
\newcommand{\Lu}{\mathcal L_1}  %Lebesgue su [0,1]
\numberwithin{equation}{section}

\newcommand{\proj}{{\rm proj}}
\newcommand{\M}{{\mathscr M}}

\title{Stability of the heat flow under convergence in concentration and  consequences}
\author{Nicola Gigli and Simone Vincini}
\date{\today}

\begin{document}
	\maketitle
	\begin{abstract}
		We extend to the framework of convergence in concentration virtually all the results concerning stability of Sobolev functions and differential operators known to be in place under the stronger measured-Gromov-Hausdorff convergence. These include, in particular:
		\begin{itemize}
		\item[i)] A general $\Gamma$--$\lims$ inequality for the Cheeger energy,
		\item[ii)] Convergence of the heat flow under a uniform $\CD(K,\infty)$ condition on the spaces.
		\end{itemize}
		As we will show, building on ideas developed for mGH-convergence, out of this latter result we can obtain clean stability statements for differential, flows of vector fields, Hessian, eigenvalues of the Laplacian and related objects in presence of uniform lower Ricci bounds. At the technical level, one of the tools we develop to establish these results is the notion  of convergence of maps between metric measure spaces converging in concentration.
		
		Previous results in the field concerned the stability of the $\CD$ (\cite{FuSh13}) and $\RCD$ (\cite{OzYo19}) conditions. The latter was obtained proving $\Gamma$-convergence  for the Cheeger energies. This  is not sufficient to pass to the limit in the heat flow; one of the byproducts of our work is the improvement of this to Mosco-convergence.
	\end{abstract}
	\tableofcontents
	\section{Introduction}
	The need for compactness when addressing problems related to curvature has been the main motivation leading to the study of nonsmooth metric (measure) spaces, following the influential ideas of Gromov (see \cite{Gr81} or \cite{Gr99} for the english translation).
	In order to produce new results with applications in the classic setting, one of the major goals of the new theory has been to generalize relevant quantities and properties and showing that the new definitions are stable (i.e. continuous) in the enlarged compactification. 
	The convergence introduced by Gromov, namely the Gromov--Hausdorff convergence, has proven to be useful in presence of a control on the sectional curvature and the dimension, yielding a limit to sequences of smooth manifolds satisfying these bound in the wider setting of Alexandrov spaces. When dealing with bounds on the dimension and the Ricci curvature, though, new phenomena appear: in order to preserve some generalized bound on the curvature, it has been proven necessary to study the asymptotic behaviour of the measures (\cite{fukaya}), leading to the notion of measured Gromov--Hausdorff convergence (the definition of which can vary a little between different authors, for some details see \cite[Section 2.4]{Gi23}).
	
	The study of Ricci curvature conditions on nonsmooth spaces has begun with the study of Ricci--limit spaces, i.e.\ measured Gromov--Hausdorff limits of smooth manifolds with uniformly bounded Ricci curvature and dimension, initiated by Cheeger and Colding (see  \cite{Ch99}, \cite{ChCo97}, \cite{ChCo00-2}, \cite{ChCo00-3}). More recently, synthetic bounds on the curvature have been introduced independently by Lott--Villani (\cite{LoVi09}) and Sturm (\cite{St06-1}, \cite{St06-2}) and have subsequently proved to be the right class in which to consider sequences of manifolds with bounds on Ricci curvature and dimension, akin to Alexandrov spaces when dealing with bounds on sectional curvature. The classes $\CD(K, N)$ (indicating $\ric\geq K$ and $\dim\leq N$) introduced by Lott, Sturm and Villani:
	\begin{itemize}
		\item contain all (complete and separable) manifolds with $\ric\geq K$ and $\dim\leq N$ when endowed with their volume measure, as well as weighted manifolds, i.e. manifolds $M^n$ with measure $e^{-V}\,\d\vol$ with $K\geq \widetilde{\ric_N}=\ric+\rm{Hess}V$ and $n\leq N$;
		\item are compact (and thus stable) with respect to mGH (measured Gromov--Hausdorff) convergence, thus containing Ricci limits.
	%	\item  behave well under many natural geometric operations, such as taking quotients and building spherical or conical suspensions.
	\end{itemize} 
	They include, though, even Finsler manifolds and Finslerian--like geometries: to deal with this problem, the first author introduced in \cite{Gi15} the so called \emph{infinitesimal Hilbertianity} property, to distinguish between genuine linear (Riemannian) and nonlinear (Finslerian) tangent structure (see also \cite{AmbrosioGigliSavare11-2}). Infinitesimally hilbertian spaces satisfying the $\CD(K, N)$ condition are called $\RCD(K, N)$ spaces.
	
	Bounds on the dimension are crucial in order to obtain compactness with respect to mGH convergence. When instead the dimension grows unbounded, different kinds of phenomena appear. What is probably the most prominent is the so called ``concentration of measure'': the fact that in some circumstances regular functions (for instance 1-Lipschitz) tend to be close to their median with high probability when the dimension is sufficiently high. This phenomenon was first described by P. Lévy (\cite{Le51}, \cite{Mi88}) and has very peculiar and at times paradoxical consequences.	A possible way to study asymptotic properties of sequences of manifolds with unbounded dimensions (which can be important e.g. when studying dimension free estimates) was sketched again by Gromov: in order to study the appearance of these ``interesting patterns'', he introduced  in \cite[Chapter $3\frac{1}{2}$]{Gr99} the notion of observable distance, strictly related the concentration of measure. In this setting, compactness is (in a somewhat generalized sense) always present (see \cite{Sh16}), while the nature of the limit spaces is not yet completely clear.
	
	In the setting of mGH convergence, the stability of many geometric and functional analytic properties of the spaces is well known and studied. A fundamental result is, as we recalled earlier, the stability of the $\CD(K, N)$ condition. This was proved in the very first papers in which this condition appeared (\cite{LoVi09}, \cite{St06-1}, \cite{St06-2}) and has its roots in the following fact, that we cite quite informally:
	\begin{thm}
		$\X_n\to \X$ with respect to mGH convergence iff $\entm{\mm_n}\xrightarrow{\Gamma} \ent$.
	\end{thm}
	Here, $\rm{Ent}$ is the so called Boltzmann entropy and the $\Gamma$--convergence has to be intended with respect to a suitable definition of weak convergence of measures. This fact was not explicitly stated in the beginning, but was at the heart of the proofs and has successively proven fruitful in the study of other stability properties. The second key result, proven in \cite{Gi10}, tells that, in presence of a uniform $\CD(K, \infty)$ bound, not only $\rm{Ent}$ $\Gamma$--converges, but also its $W_2$-gradient flow, that according to the seminal paper \cite{JKO98} can be interpreted as heat flow,  does. This convergence, coupled with the fact that the heat flow can also be seen as $L^2$-gradient flow of the Cheeger energy even in this nonsmooth setting (\cite{Gigli-Kuwada-Ohta10}, \cite{AmGiSa14}) allows to prove the convergence of Cheeger energy and in particular the stability of the $\RCD$ condition (\cite{AmGiSa14},\cite{Gi15},\cite{GiMoSa15}). Finally, under appropriate conditions (e.g. some uniform log--Sobolev inequalities), even the spectrum of the Laplacian has been proven to be stable. Building on the convergence of the Cheeger energy, one can also define a notion of convergence for vector and even tensor fields and prove the stability of flows of vector fields (\cite{Honda11-2}, \cite{Gigli14}, \cite{AH16}).
	
	In the more general setting of convergence in concentration, only partial results were present, coming from the work of the Japanese school. The stability of the $\CD(K, \infty)$ condition was proved in \cite{FuSh13}:
	\begin{thm}[\cite{FuSh13}]
		Let $(\X_n, \sfd_n, \mm_n), (\X, \sfd, \mm)$ be n.m.m.s. with $\X_n\to \X$ in concentration. If $\X_n$ is $\CD(K, \infty)$ for all $n\geq 1$ and some $K\in\mathbb R$, then $\X$ is $\CD(K, \infty)$.
	\end{thm}
	Notice that this result is not weaker than the corresponding one in the setting of mGH convergence: if the dimension is uniformly bounded, the sequence is mGH compact and thus converges to $\X$ with respect to mGH topology. In \cite{OzYo19}, Ozawa and Yokota proved a stability result for Cheeger energy, obtaining the stability of $\RCD(K, \infty)$ condition:
	\begin{thm}[\cite{OzYo19}]
		Let $(\X_n, \sfd_n, \mm_n), (\X, \sfd,\mm)$ be n.m.m.s. with $\X_n\to \X$ in concentration and with $\X_n$ satisfying $\CD(K, \infty)$ for all $n\geq 1$. Then $\ch_n \xrightarrow{\Gamma} \ch$ with respect to strong $\L^2$ convergence. In particular, if $\X_n$ is $\RCD(K, \infty)$ for all $n\geq 1$ then $\X$ is as well.
	\end{thm}
	%The $\Gamma$--convergence of Cheeger energies with respect to strong $\L^2$ convergence is not enough to obtain convergence of the heat flows, so all the other results proven for mGH in \cite{GiMoSa15} were up to now not known in this context.	\newline

	In this paper we prove that, as in the mGH case, the slope of the entropy $|\partial^- \rm{Ent}|$ $\Gamma$--converges as the underlying spaces converge in concentration and in presence of a uniform bound on the Ricci curvature. 
	While the structure of the proofs is heavily based on \cite{Gi10}, there are important technical differences, among which the impossibility of an ``extrinsic'' realization of convergence in concentration.
	
	 In Section \ref{section: preliminari} we recall some basic definitions and results which will be of use, among which the definition and main properties of convergence in concentration. 
	 
	 In Section \ref{section: conv} we propose a systematization of the notions of convergence of measures and functions defined on a sequence of spaces which converge in concentration. The definitions of weak convergence of measures and (weak and strong) $\L^2$ convergence were firstly given in \cite{OzYo19}, inspired by \cite{KuSh03} and \cite{GiMoSa15}. In this section we introduce a natural notion of equi--integrability for measures and extend the definitions of $\L^p$ convergences \cite{OzYo19} to the general exponent $p\in [1, \infty)$. Moreover, we introduce a notion of convergence in measure, which allows to characterize $\L^p$ convergence in a ``dominated--convergence'' fashion and to exploit ideas from the theory of Young measures, akin to what has been done in \cite{GiMoSa15}. The geometric nature of this notions is proved in Section \ref{section: technical}, where it is shown that different realizations of the spaces in question yield the same limit structure to sequences of measures and functions, up to suitable automorphisms.

	 In the following Section \ref{section: limsup} we show that, as in the mGH setting, the $\Gamma-\lims$ inequality for the slope of the entropy always holds, regardless of curvature bounds.
	 
	 Section \ref{section: gamma conv slope} contains the core of the paper, the $\Gamma-\limi$ of the slope:
	 \begin{thm}\label{teo: liminf slope}
	 	Let $\X_n$, $n\geq 1$, and $\X$ be n.m.m.s. such that $\X_n\to \X$ in concentration. Assume that $\X_n$ satisfy $\CD(K, \infty)$ for some $K\in\mathbb R$ (and thus $\X$). Let $\mu_n\in \pr(\X_n)$ be equi-integrable and such that $\mu_n\to\mu$ for some $\mu\in \pr(\X)$. Then 
	 	\begin{equation}
	 		\slope{\mm}{\mu}\leq \limi_n \slope{\mm_n}{\mu_n}.
	 	\end{equation}
	 \end{thm}
 	 The idea of the proof is simple: the uniform $K$--geodesical convexity of the entropies $\entm{\mm_n}$ allows to regard the slope as the supremum of a difference quotient (up to a correction): then, in order to prove the $\limi$ inequality, it is sufficient to be able to approximate the difference quotient itself. To do so means being able, given a converging sequence $\pr(\X_n)\ni\mu_n\rightharpoonup \mu\in\pr (\X)$ and a competitor $\nu\in \pr(\X)$ in the incremental ratio, to build another sequence $\nu_n\in \pr(\X_n)$ such that $\entm{\mm_n}(\nu_n)\to \ent(\nu)$ and $\W_2(\mu_n, \nu_n)\to \W_2(\mu, \nu)$. The careful construction of such a sequence is again topic of Section \ref{section: technical}, in Subsection \ref{subsection: pullback guidato}. In the rest of Section \ref{section: gamma conv slope} we exploit this result to obtain a direct proof of the full $\Gamma$--convergence of $|\partial^-\rm{Ent}|$, the convergence of the gradient flow of $\rm{Ent}$ and a Rellich--type theorem which, in turn, leads to the Mosco convergence of Cheeger energies and the convergence of the heat flow.
 	 
 	 Finally, in Section \ref{section: app} we show some of the possible applications of the results of the previous sections. In particular we show how it is possible, as in the mGH case:
 	 \begin{itemize}
 	 	\item to recover the stability of the eigenvalues of the laplacian;
 	 	\item to define a notion of $W^{1, 2}$ convergence for functions on varying spaces and a notion of $\L^2$ convergence of vector and tensor fields;
 	 	\item to recover the stability of the solutions to the continuity equations.
 	 \end{itemize}
	
	%  Finally, in the Appendix we show how, if convergence in concentration is realized through a suitable choice of projections (cfr. Section \ref{section: preliminari}), the $\Gamma-\lims$ for the slope of the entropy always holds, regardless of curvature bounds. The imposition in the choice of projections seems technical and we believe it can be dropped.
	
	\section{Preliminaries} \label{section: preliminari}
	\subsection{Some notation}
	The main objects of study in this paper are normalized metric measure spaces, whose properties come from the distance structure, the probability measure and their interaction. Here we recall some definitions and fix the notation for the rest of the work.
	
	Given a complete and separable metric space $(\X,\sfd)$, we denote by $C^0_b(\X)$ the set of continuous and bounded functions on $\X$, by $\B(\X)$ the family of Borel sets on $\X$ and by $\M(\X)$ the set of finite Borel measures on $\X$. On $\M(\X)$ we consider the weak convergence defined in duality with bounded continuous functions, i.e. given a sequence $(\mu_n)_{n\geq 1}$ and a limit measure $\mu$ in $\M(\X)$, we say that $\mu_n$ converges weakly to $\mu$ (or $\mu_n\rightharpoonup \mu$) if 
	\begin{equation}
		\int_\X \varphi\, \d\mu_n\to \int_\X \varphi\, \d\mu
	\end{equation}
	for all $\varphi\in C^0_b(\X)$.
	The weak topology coincides with what is usually called \emph{narrow topology} when restricted to $\pr(\X)$, the set of Borel probability measures. On this set, the weak topology can be metrized by either the Prohorov distance or the $p$-Wasserstein distance, $p\in[1,\infty)$, induced by the truncated metric $\sfd\wedge 1$.
%	$\W_p$:
%	\begin{deff}
%		Let $(\X, d)$ be a metric space. The \emph{Prohorov distance} between two probability measures $\mu, \nu\in \pr(\X)$ is defined as follows:
%		\begin{equation}
%			d_{P}(\mu, \nu)\coloneqq\inf\{\varepsilon>0:\,\nu(B)\leq \mu(B_\varepsilon)+\varepsilon \quad \forall B\in \mathcal{B}(\X)\},
%		\end{equation}
%		where $B_\varepsilon=\{x\in \X:\, d(x, B)\leq \varepsilon\}$.
%	\end{deff}
We shall frequently use the Prohorov precompactness criterion:
	\begin{thm}[Prohorov]
		Let $\mathcal K\subset \pr(\X)$. Then $\mathcal K$ is precompact with respect to weak convergence iff it is tight, i.e.\ for all $\varepsilon>0$ there exists a compact set $K\subset \X$ such that $\mu(\X\setminus K)<\varepsilon$ for all $\mu \in \mathcal K$.
	\end{thm}
		\begin{deff}
		A \emph{metric measure space} (short \emph{m.m.s.}) is a triple $(\X, \sfd, \mm)$ where $(\X, \sfd)$ is a complete and separable metric space and $\mm$ is a nonnegative Borel measure which is finite on bounded sets. When $\mm\in\pr(\X)$, we say that the m.m.s. is normalized (or a n.m.m.s.).
	\end{deff}
		We say that two m.m.s.\ are isomorphic if there exists a measure preserving isometry of the supports of the measures. Notice in particular that $(\X,\sfd,\mm)$ is always isomorphic to $(\supp \mm,\sfd,\mm)$.	With a small abuse in notation, we will often identify isomorphic metric measure spaces.

%	In the following we will write $d^H_{\L^0}$ for the Hausdorff distance induced by $d_{\L^0}$ on the closed subsets of $\L^0$, i.e. given $A, B\subset \L^0(\X)$ closed, we write
%	\begin{equation}
%		d^H_{\L^0}(A, B)\coloneqq \inf \{\varepsilon>0: A\subset B+B(0, \varepsilon), B\subset A+B(0, \varepsilon)\},
%	\end{equation}
%	where the balls are those coming from the $\L^0$--distance.
	
%	In order to study the convergence of (n.)m.m.s. it is useful to introduce some well known distances: the \emph{Prohorov distance}, metrizing the weak convergence of probability measures and the $\L^0$--\emph{distance} (or \emph{Ky--Fan distance}), metrizing the convergence in measure.	

	\subsection{Convergence in concentration of metric measure spaces}

	Given a normalized metric measure space $(\X,\sfd,\mm)$, we denote by $\L^0(\X)$ the set of $\mm$--a.e.\ equivalence classes of real valued functions. This space is equipped with the complete and separable distance $\sfd_{\L^0}$ defined as
		\begin{equation}
			\sfd_{\L^0}(f, g)\coloneqq \inf\{\varepsilon>0:\, \mm(\{|f-g|>\varepsilon\})<\varepsilon\}.
		\end{equation}
	It is well known that $f_n\to f$  in $\L^0(\X)$, also called in measure or probability, if and only if from each subsequence  one can extract a further subsequence which converges to $f$ pointwise $\mm$--a.e.. We shall also write   $\sfd^{\sf H}_{\L^0}$ for the Hausdorff distance induced by $\sfd_{\L^0}$ on the subsets of $\L^0$, i.e. given $A, B\subset \L^0(\X)$, we write
	\begin{equation}
		\sfd^{\sf H}_{\L^0}(A, B)\coloneqq \inf \big\{\varepsilon>0: A\subset B^\eps,\ B\subset A^\eps \big\},
	\end{equation}
	where the $\eps$-neighbourhoods  $A^\eps,B^\eps$ of the sets $A,B$  are defined via the $\L^0$--distance.

We shall denote by $I$ the interval $[0,1]\subset\R$, equipped with its standard distance and measure. A {\bf parameter} for the n.m.m.s.\ $(\X, \sfd, \mm)$ is a Borel map $\varphi:I\to \X$ such that  $\varphi_\# \Lu=\mm$. It is easy to see that parameters exist for any n.m.m.s. (this is also a consequence of the more powerful Borel isomorphism theorem, see for instance \cite[Theorem 17.41]{Ke95}).

Given a metric space $(\X,\sfd)$, we shall denote by   $\Lip_1(\X)$ the space of real valued 1--Lipschitz functions on $\X$.

	\begin{deff} \label{def: convergenza in concentrazione}
		Let $(\X, \sfd_\X, \mm_\X), (\Y, \sfd_\Y, \mm_\Y)$ be n.m.m.s.. The \emph{observable distance} between $\X$ and $\Y$ is defined as follows:
		\begin{equation}
			\sfd_{\sf conc}(\X, \Y)\coloneqq\inf\Big\{\sfd_{\L^0}^{\sf H}(\varphi^*\Lip_1{(\X)}, \psi^*\Lip_1{(\Y)}): \ \varphi, \psi\text{ parameters for $\X,\Y$ respectively} \Big\}.
		\end{equation}
		If $\sfd\sf_{conc}(\X_n, \X)\to 0$ we say that $\X_n\xrightarrow{conc}\X$ or that $\X_n$ converges to $\X$ \emph{in concentration}.
	\end{deff}
	In particular, if the limit space $\X$ is a point space, we say that the sequence $(\X_n)$ is a Lévy family. 
	
	An  equivalent definition of distance in concentration can be given using couplings of the measures of the two spaces, so that the projections take the place of parameters. More precisely, given $\ppi\in \adm(\mm_\X, \mm_\Y)$ - the set of probability measures on $\X\times\Y$ having $\mm_\X,\mm_\Y$ as marginals - we can easily estimate from above the observable distance with
	\begin{equation}
		\sfd_{\sf conc}^\sppi(\X, \Y)\coloneqq \sfd_{\L^0}^{\sf H}\big(\operatorname{proj}_1^*\Lip_1{(\X)}, \operatorname{proj}_2^*\Lip_1{(\Y)}\big),
	\end{equation}
	where the $\L^0$ distance on $\X\times \Y$ is computed with respect to the measure $\ppi$. This estimate can be made sharp by choosing the appropriate coupling:
	\begin{prop}[\cite{Na22}] We have
		\begin{equation}
			\sfd_{\sf conc}(\X, \Y)=\min_{\ppi\in \adm(\mm_\X, \mm_\Y)} \sfd_{\sf conc}^\sppi(\X, \Y).
		\end{equation}
	\end{prop}
	
%	Convergence in concentration is a weaker notion than mGH convergence: indeed, it holds that $d_{conc}\leq \square$, so it allows for more converging sequences. It is also possible to give another characterization based directly on $\square$: given a n.m.m.s. $\X$ we denote $\mathcal P_\X\coloneqq\{\Y\prec \X: \Y\text{ is a n.m.m.s.}\}$, with $\Y\prec \X$ meaning that a measure--preserving 1-Lipschitz map $\phi:\X\to \Y$ exists. 
%	\begin{prop}[{\cite[Proposition 6.12]{Sh16}}]\label{proposition: piramidi}
%		$\X_n\xrightarrow{conc}\X$ iff $\mathcal P_{\X_n}\to \mathcal P_\X$ weakly Hausdorff, i.e.
%		\begin{itemize}
%			\item For each $\Y\in \mathcal P_\X$ there exists $\Y_n\in\mathcal P_{\X_n}$, $n\geq 1$ s.t. $\square(\Y_n, \Y)\to 0$.
%			\item For each $\Y\notin \mathcal P_\X$ and sequence $(\Y_n)_{n\geq 1}$ with $\Y_n\in\mathcal P_{\X_n}$, it holds $\limi_n \square(\Y_n, \Y)>0$.
%		\end{itemize}
%	\end{prop}
%	\begin{rmk}
%		In \cite{Sh16} it is also proven that the identification of $\X$ with $\mathcal P_\X$ is a topological embedding onto a suitable pre-compact (with respect to weak Hausdorff convergence) subset of the parts of the space of n.m.m.s..
%		Recently, a similar results has been obtained in the context of Gromov--Hausdorff convergence by Nakajima and Shioya in \cite{NaSh21}, where this characterization has been used in concert with the notion of ultralimit to provide a compactification of the space of compact metric spaces.
%	\end{rmk}
	A very useful characterization of sequences of spaces converging in concentration can be given in terms of ``projection'' maps on the limit space. This procedure allows to regard the converging spaces as sort of fibrations on the limit.
	\begin{prop}[\cite{FuSh13}, \cite{OzYo19}]\label{prop: caratterizzazione}
		Let $(\X_n, \sfd_n, \mm_n)$, $n\geq 1$, $(\X, \sfd, \mm)$ be n.m.m.s.\ such that $\X_n\concto \X$. Then there exist $p_n:\X_n\to \X$ Borel maps, $K_n\subset \X_n$ compact subsets and a sequence $\varepsilon_n\downarrow 0$ such that:
		\begin{enumerate}[i)]
			\item\label{it:pnconc} $\sfd_{\L^0}^{\sf H}(\Lip_1{(\X_n)}, p_n^*\Lip_1{(\X)})\leq \varepsilon_n$;
			\item\label{it:convpn} $\left(p_n\right)_{\#} \mm_n \rightharpoonup \mm$;
			\item\label{it:appr1lip} $\sfd(p_n(x), p_n(x'))\leq \sfd_n(x, x')+\varepsilon_n$ for all $x, x'\in K_n$;
			\item\label{it:misurakn} $\mm_n(K_n)\geq 1-\varepsilon_n$;
			\item\label{it:bounded} $\lims_n \sup_{x\in \X_n\setminus K_n} \sfd(p_n(x), y)<\infty$ for every  $y\in \X$.
		\end{enumerate}
	\end{prop}
	\begin{rmk}\label{oss: equi}
		Notice that the \emph{non-exceptional domains} $K_n$ can be shrunk arbitrarily as long as $\mm_n(K_n)\to 1$ still holds. Also,  the value of $p_n$ outside $K_n$ is not really relevant as long as the requirement in item \eqref{it:bounded} is satisfied; one could, for instance, define them to be identically $\bar y$ outside $K_n$, where $\bar y$ is some fixed point in $\X$. 
	\end{rmk}
	This last result allows a more hands--on kind of approach to the study of convergence in concentration, which is the one that we will adopt in the rest of this work.
\begin{rmk}[Comments about the ``projection'' maps $p_n$]
One can roughly describe the content of Proposition \ref{prop: caratterizzazione} as follows:
\begin{itemize}
\item[a)] The spaces $\X_n$ are divided into a `relevant part' $K_n$ (whose measure is approaching 1 by item \eqref{it:misurakn}) and an `irrelevant part' $\X_n\setminus K_n$ that we don't really care about.
\item[b)] There are maps $p_n:K_n\to \X$ that are approximately 1-Lipschitz (by item \eqref{it:appr1lip})  and that almost send $\mm_n$ to $\mm$  (by item \eqref{it:convpn}).
\item[c)] The combination of the above and item \eqref{it:pnconc} tell that on the fibers of the $p_n$'s we basically observe a concentration of measure phenomenon. 
\end{itemize}
This last point is made precise by the following result, proved in \cite{Sh16}. To state it we recall that given $\alpha\leq\mm(\X)$ the \emph{partial diameter} $\operatorname{Diam}(\X; \alpha)$ of $\X$ with parameter $\alpha$ is defined as 
		\begin{equation}
			\operatorname{Diam}(\X; \alpha)\coloneqq \inf \{ \operatorname{diam}(A)\ :\  A\subset\X\text{ is Borel with } \mm(A)\geq \alpha\}.
		\end{equation}
Then the \emph{observable diameter} $\operatorname{ObsDiam}(\X, \alpha)$ of $\X$ with parameter $\alpha$ is defined as
		\begin{equation}
			\operatorname{ObsDiam}(\X, \alpha)\coloneqq\sup\{ \operatorname{diam}(f_\#\mm; \alpha)\ :\  f\in\Lip_1{(\X)}\},
		\end{equation}
		where by $\operatorname{diam}(f_\#\mm; \alpha)$ we intend the partial diameter of $\R$ equipped with the measure $f_\#\mm$.
It is clear that $\operatorname{ObsDiam}(\mm, \alpha)\leq \operatorname{Diam}(\mm; \alpha)$, but less so that the inequality can be strict: the key example to have in mind is that of the spheres $S^n$ equipped with their standard round metric and the normalized volume measure. One can check - see \cite[Section 3$\tfrac12$.20]{Gr99} -  that in this case we have 
\[
\text{$\operatorname{Diam}(S^n; \alpha)\to\tfrac\pi2\quad$ and $\quad\operatorname{ObsDiam}(S^n; \alpha)\to 0\quad$ as $n\uparrow\infty$ for any $\alpha\in(0,1)$.}
\]

More generally, we have $\operatorname{ObsDiam}(\X_n; \alpha)\to 0$ for any $\alpha\in(0,1)$ if and only if the sequence $(\X_n)$ converges in concentration to the one point space (see \cite{Sh16}).

Now, the rigorous claim behind item $(c)$ above is: for any Borel set $A\subset \X$ and any $\alpha\in(0,\mm(A))$ we have
\[
\lims_n \operatorname{ObsDiam}(p_n^{-1}(A),\alpha)\leq \operatorname{diam}(A),
\]
where it is intended that $p_n^{-1}(A)\subset\X_n$ is equipped with the restriction of the distance and the measure. Notice that this inequality goes, in some sense, in the opposite direction of the almost 1-Lipschitz property encoded in item \eqref{it:appr1lip}.
\end{rmk}
		
\begin{rmk}\label{oss: equivalenza}
Let the maps $p_n$ be as in Proposition \ref{prop: caratterizzazione} and let $q_n:\X_n\to\X$ be such that 
\begin{equation}
\label{eq:eqpnqn}
\lims_n\sfd_{\L^0}(p_n,q_n)= 0,
\end{equation}
where 
\begin{equation}\label{eq: dist L0}
	\sfd_{\L^0}(p_n, q_n)\coloneqq \inf\{\varepsilon>0:\, \mm_n(\{\sfd(p_n, q_n)>\varepsilon\})<\varepsilon\}.
\end{equation}
Then it is easy to prove that the $q_n$'s also satisfy the properties in the statement (possibly with different $K_n$'s and $\eps_n$'s).

Whenever \eqref{eq:eqpnqn} holds we say that the $p_n$'s and the $q_n$'s are \emph{equivalent}.
\end{rmk}

	\section{Convergence of measures and functions}\label{section: conv}
	In order to study the stability of quantities defined on n.m.m.s. with respect to convergence in concentration it is fundamental to be able to speak of convergence of measures and functions defined along converging sequences of spaces. These concepts are well known in the case of mGH convergence (see \cite{Gi23} and references therein) and were partly sketched in \cite{OzYo19} in the setting of convergence in concentration. Since we will need to rely on somewhat finer properties of the convergences, we propose here a more systematic treatment of functional limits in this setting, extending the definitions already present in the mGH case.
	
	In the rest of Section \ref{section: conv} we will consider a sequence of n.m.m.s. $(\X_n, \sfd_n, \mm_n)$, a \emph{base} n.m.m.s. space $(\X, \sfd, \mm)$ and a sequence of Borel \emph{projections} $p_n:\X_n\to \X$ such that $\left(p_n\right)_\#\mm_n\rightharpoonup \mm$. For ease of notation, we set $\tilde{\mm}_n=\left(p_n\right)_\#\mm_n$. Notice that, while we are interested in the situation in which the spaces $\X_n$ do converge in concentration, the structure of the convergences depends only on this more general form of the spaces. We want to stress the fact that all of the convergences that will be studied depend on the choice of the maps $p_n$. We will later prove in Section \ref{section: technical} that if the spaces $\X_n$ are converging in concentration and the projections $p_n$ are as in Proposition \ref{prop: caratterizzazione}, then the notion of convergence is somehow canonical.
	
	In order to keep the notation as simple as possible, anyway, whenever the maps - and thus the convergence structure - will be clear from the context, we will omit the explicit reference to the $p_n$'s.

	\subsection{Convergence of measures}
	We start by defining what we mean by convergence of measures.
	\begin{deff}[Weak convergence of measures]
		Let $\mu_n$ be finite Borel measures on $\X_n$, $n\geq 1$, and let $\mu$ be a finite Borel measure on $\X$. We say that $\mu_n\rightharpoonup\mu$ whenever $\left(p_n\right)_\#\mu_n \rightharpoonup\mu$.
	\end{deff}
	Remark \ref{oss: equi} suggests that the notion of weak convergence of measures depends way too much on the particular structure of the projections $p_n$ even when the spaces converge in concentration. In order to tackle this problem is to require that the sequence of measures does not weigh too much sets of vanishing measure. The correct way to do this, as Theorem \ref{teo: unicità limite} will show, is to ask for equi--integrability.
	\begin{deff}[Equi--integrability]\label{def: equi--int}
		We say that the family $\mathcal F\subset \bigcup_n \M(\X_n)$ is \emph{equi--integrable} if 
		\begin{equation}
			\lim_n \sup_{\mu\in\mathcal{F}\cap \mathscr M(\X_n)} |\mu|(E_n)=0
		\end{equation}
		for each sequence of Borel sets $E_n\subset \X_n$ such that $\mm_n(E_n)\to 0$.
	\end{deff}
% {\color{red}	
% qui il punto e' vedere se riusciamo, col trucco dei `molti bump' a trovare l'isomorfismo che non dipende dalla successione di misure
% \begin{rmk}
% 		Assume $\X_n\xrightarrow{conc} \X$ and the convergence is realized (in the sense of Proposition \ref{prop: caratterizzazione}) by two sequences of projections $(p_n)_{n\geq 1}$ and $(q_n)_{n\geq 1}$. Then if $(\mu_n)_{n\geq 1}$ is an equi--integrable sequence converging to $\mu_p$ with respect to the structure induced by $(p_n)_{n\geq 1}$ and to $\mu_q$ with respect to $(q_n)_{n\geq 1}$, it is easy to see that there exists an isometry $\phi: \supp \mu_p\to \supp \mu_q$ such that $\phi_{\#} \mu_p=\mu_q$. Indeed equi--integrability implies that $(\X_n, d_n, \mu_n)$ converges in concentration to both $(\X, \sfd, \mu_p)$ and $(\X, d, \mu_q)$, which must then be isomorphic.
% 		Notice however that the isomorphism $\phi$ depends a priori on the sequence $(\mu_n)_{n\geq 1}$, while it would be nice to have it depending only on the sequences of projections.
% 	\end{rmk}}
	For later use, we prove a couple of easy lemmata. In the following we will say that a sequence of measures $(\mu_n)_{n\geq 1}$ has a weakly convergent subsequence if there exists a subsequence $(\mu_{n_j})_{j\geq 1}$ which converges weakly with respect to the projections $p_{n_j}$. It is easy to check that a sequence converges to a measure if and only if all of its subsequences have a subsequence converging to that measure.
	\begin{lm}\label{lemma: compattezza equi-int}
		Any equi--integrable sequence of probability measures $n\mapsto\mu_n\in \pr(\X_n)$ has a weakly convergent subsequence.
	\end{lm}
	\begin{proof}
		It is immediate to check that $\left(\left(p_n\right)_\#\mu_n \right)_{n\geq 1}$ is tight.
	\end{proof}
	\begin{rmk}
		The same holds if the measures are not required to be probability but at least $\sup_n \|\mu_n\|_{\rm{TV}}<\infty$.
	\end{rmk}
	Another useful result is a this sort of continuity of any equi--integrable sequence.
	\begin{lm}\label{lemma: equi-int cont}
		Let $\mu_n\in \Mp(\X_n)$ be an equi--integrable sequence and let $\varepsilon>0$. Then there exist $\delta>0$ such that whenever $E_n\subset \X_n$ is a sequence of Borel sets with $\lims_n \mm(E_n)\leq \delta$, then $\lims_n \mu(E_n)\leq \varepsilon$.
	\end{lm}
	\begin{proof}
		The proof is a standard contradiction argument.
	\end{proof}
	\begin{lm}\label{lemma: troncamenti}
		Let $\mu_n=f_n \mm_n\in \mathcal \pr(\X_n)$ be an equi--integrable sequence. Then for every $M>0$
		\begin{equation}
			\limi_n\int_{\X_n} \left(f_n\wedge M\right)\,\d\mm_n>0.
		\end{equation}
		Moreover
		\begin{equation}\label{eq: limite troncamenti}
			\lim_M \limi_n \int_{\X_n}\left(f_n\wedge M\right)\,\d\mm_n=1
		\end{equation}
	\end{lm}
	\begin{proof}
		By contradiction, assume the thesis is false. Since 
		\begin{equation}
			\int_{\X_n} \left(f_n\wedge M\right)\,\d\mm_n= \mu_n(\{f_n<M\})+M\mm_n(\{f_n\geq M\}),
		\end{equation}
		we must have 
		\begin{equation}
			\lims_n \mm_n(\{f_n\geq M\})=\lims_n \mu_n(\{f_n<M\})= 0.
		\end{equation} 
		By definition of equi--integrability, then $\lims\mu_n(\{f_n\geq M\})= 0$, which is absurd since $\mu_n(\X)\equiv 1$.
		In order to prove \eqref{eq: limite troncamenti}, it is enough to notice that, by the same computations as above,
		\begin{equation}
			0\leq 1-\lim_M \limi_n \int_{\X_n} \left(f_n\wedge M\right)\,\d\mm_n\leq \lim_M\lims_n \mu_n(\{f_n\geq M\}). 
		\end{equation}
		The vanishing of the last term easily follows from Markov inequality and Lemma \ref{lemma: equi-int cont}.
	\end{proof}

	For later use it is important to notice the following semicontinuity result: if $\X_n\concto\X$, $(p_n)_{n\geq 1}$ is a sequence of projections as in Proposition \ref{prop: caratterizzazione}, $n\to\mu_n\in \pr(\X_n)$, $n\to\nu_n\in \pr(\X_n)$ are equi--integrable sequences such that $\mu_n\rightharpoonup\mu$ and $\nu_n\rightharpoonup\nu$, then
	\begin{equation}\label{eq: semicont W2}
		W_2(\mu, \nu)\leq \limi_n W_2(\mu_n, \nu_n).
	\end{equation}
	This follows from \cite[Lemma 35]{OzYo19} up to a minor modification of the proof.
	
	\subsection{Convergence of functions}
	Building on the definition of convergence of measures, we can introduce functions and their convergences. The first notion that we present is convergence in measure. The definition we chose is inspired by Young measures and mimics the one used in \cite{AmStTr17}.
	\begin{deff}[Convergence in measure]\label{def: conv L0}
		Let $(\Z, \sfd_\Z)$ be a complete and separable metric space. Let $T_n:\X_n\to \Z$, $n\geq 1$ and $T:\X\to \Z$ be Borel maps. We say that $T_n$ converges in measure to $T$, or $T_n\xrightarrow{\L^0}T$ if 
		\begin{equation}
			\Phi(T_n)\,\mm_n\rightharpoonup \Phi(T)\,\mm
		\end{equation}
		for all $\Phi\in C^0_\b(\Z)$.
	\end{deff}
%	\begin{lm}\label{lemma: cs conv misura}
%		Let $(\Z, d_\Z)$ be a complete separable metric space. Let $T_n: \X_n\to \Z$ be Borel maps and let $T_\X\to \Z$ Borel. Then, if $(T_n)_\# \mu_n\to T_\#\mu$ for all sequences $\mu_n\rightharpoonup\mu$ with $\mu_n\leq C m_n$, it follows that $T_n\to T$ in $\L^0$.
%	\end{lm}
%	\begin{proof}
%		Notice that testing the hypothesis with $\mu_n\coloneqq (\varphi\circ p_n) m_n$ and $Phi\in \mathcal C^0_b(\Z)$ one can directly recover the definition.
%	\end{proof}
	In order to work with convergence in measure (and to show why the naming choice is natural), it is convenient to introduce an useful characterization, directly inspired from \cite{Gi23} and based on optimal transport. The idea is to chose suitable couplings of $\mm_n$ and $\mm$ in order to directly compare functions on the converging space and functions on the limit space.
	\begin{deff}
		Let $\aalpha_n\in \adm(\mm_n, \mm)$ for $n\geq 1$. We say that $(\aalpha_n)_n$ are \emph{good couplings} if
		\begin{equation}
			\int_{\X\times \X} \left(\sfd_\X\circ (p_n, {\rm id})\wedge 1\right)\,\d\aalpha_n\to 0
		\end{equation} 
		or, equivalently, $\int_{\X\times \X} \left(\sfd_\X\wedge 1\right)\,\d\tilde{\aalpha}_n\to 0$, where $(p_n, id)_\#\aalpha_n=\tilde{\aalpha}_n$. 
	\end{deff}
	 The choice of good couplings allows to reduce the $\L^0$ convergence to the vanishing of a sequence of ``distances''.
	\begin{prop}\label{prop: L0}
		Let $(\Z, \sfd_\Z)$ be a complete and separable metric space. Let $T_n: \X_n\to \Z$ be Borel maps, $n\geq 1$, $(\aalpha_n)_n$ good couplings and $T:\X\to \Z$ a Borel map. Then $T_n\xrightarrow{\L^0}{}T$ iff 
		\begin{equation}\label{eq: L0 conv}
			\lim_n\aalpha_n\left(\left\{\sfd_\Z\left(T_n\circ\operatorname{proj}_1, T\circ\operatorname{proj}_2\right)>\varepsilon\right\}\right)=0
		\end{equation}
		for all $\varepsilon>0$.
	\end{prop}
	\begin{proof}
		Suppose that \eqref{eq: L0 conv} holds and fix $\Phi\in C^0_\b(\Z)$. Notice that by boundedness of $\Phi$ it is enough to test the weak convergence on Lipschitz functions, so let us fix $\varphi\in \Lip(\X)$. It holds that
		\begin{align}
			\bigg|\int_\X \varphi(x)\, \d\left(p_n\right)_\#&(\Phi(T_n)\mm_n)-\int_\X \Phi(T)\varphi\,\d\mm \bigg|\\
			&=\left|\int_{\X_n\times \X}(\Phi(T_n(x))\varphi(p_n(x))-\Phi(T(y))\varphi(y))\, \d\aalpha_n(x, y)\right|\\
			&\leq \|\varphi\|_{\infty}\int_{\X_n\times \X}|\Phi(T_n(x))-\Phi(T(y))|\,\d\aalpha_n(x, y)\\
			&\quad \quad +\|\Phi\|_{\infty}(2\|\varphi\|_{\infty}\vee\Lip(\varphi)))\int_{\X\times \X}1\wedge \sfd_\X(x, y)\,d\tilde{\aalpha}_n(x, y).
		\end{align}
		The latter term vanishes in the limit, due to the choice of $\aalpha_n$, so it suffices to estimate the former. 
		Fix $\varepsilon>0$. By tightness of $T_\# \mm$, there exist $K\subset \Z$ compact such that $\mm(\{T(x)\notin K\})<\varepsilon$. By compactness of $K$ and continuity of $\Phi$, there exists $0< \delta\leq \varepsilon$ such that $|\Phi(z_1)-\Phi(z_2)|\leq \varepsilon$ whenever $z_1\in K$ and $\sfd_\Z(z_1, z_2)\leq \delta$. Now set $E_n\coloneqq\{T(x)\in K, \sfd_\Z(T_n(x), T(y))\leq \delta\}$. In particular, $\limsup_n \aalpha_n(\X_n\times \X\setminus E_n)\leq \varepsilon$.
		We are now able to estimate the first integral as follows:
		\begin{align}
			\int_{\X_n\times \X}&|\Phi(T_n(x))-\Phi(T(y))|\,\d\aalpha_n(x, y)\leq  2\|\Phi\|_\infty \aalpha_n(\X_n\times \X\setminus E_n)\\
			&+\int_{E_n} |\Phi(T_n(x))-\Phi(T(y))|\,\d\aalpha_n(x, y)\\
			&\leq 2\|\Phi\|_\infty \aalpha_n(\X_n\times \X\setminus E_n)+\varepsilon.
		\end{align}
		Sending $\varepsilon$ to 0, also this term vanishes.
		
		Assume now that \eqref{eq: L0 conv} is not verified and let us prove that $T_n$ does not converge in measure to $T$.
		By assumption there exist $\varepsilon, \delta>0$ such that
		\begin{equation}
			\aalpha_n\left(\left\{\sfd_\Z\left(T_n\circ\operatorname{proj}_1, T\circ\operatorname{proj}_2\right)>\varepsilon\right\}\right)>\delta
		\end{equation}
		frequently in $n$. Up to taking a smaller $\delta$, by the tightness of $T_\# m$ we may assume that there exist $z\in \Z$ and $E_n\subset\{\sfd_\Z(T_n\circ\operatorname{proj}_1, T\circ\operatorname{proj}_2)>\varepsilon\} \cap \proj_1^{-1}\{\sfd_\Z(z, T)\leq \frac{\varepsilon}{4}\}$ and such that $\aalpha_n(E_n)>\delta$ frequently.
		Moreover, let $0\leq\Phi\leq 1$ be such that $\Phi(x)= 1$ when $\sfd_\Z(x, z)<\frac{\varepsilon}{4}$, $\Phi(x)=0$ when $\sfd_\Z(x, z)>\frac{\varepsilon}{2}$.
		In order to conclude, let us prove that $\Phi(T_n)\mm_n$ does not converge to $\Phi(T)\mm$. In order to do so, we choose an appropriate $\varphi$ test for which the coupling does not converge.
		Let $\varphi\in \Lip(\X)$ be such that $0\leq \varphi\leq 1$ and $\|\varphi-\mathbbm{1}_{\{\sfd_\Z(T, z)<\frac{\varepsilon}{4}\}}\|_{\L^1(m)}\leq \frac{\delta}{4}$.
		We are to estimate the following:
		\begin{align}
			\bigg|\int_\X \Phi(T)\varphi\,\d\mm-\int_{\X_n}&\Phi(T_n)\varphi\circ p_n\,\d\mm_n\bigg|\\
			&=\left|\int_{\X_n\times \X} (\Phi(T(y))\varphi(y)-\Phi(T_n(x))\varphi(p_n(x)))\,\d\aalpha_n\right|.
		\end{align}
		Notice that the second term is equal, up to a negligible rest to
		\begin{equation}\label{eq: conv mis}
			\left|\int_{\X_n\times \X} \left[\Phi(T(y))-\Phi(T_n(x))\right]\varphi(y)\,\d\aalpha_n\right|:
		\end{equation}
		indeed, the difference can be estimated as
		\begin{equation}
			\int_{\X_n\times \X} \Phi(T_n(x)) \left|\varphi(p_n(x))-\varphi(y)\right|\,\d\aalpha_n\leq(2\|\varphi\|_\infty\vee\Lip(\varphi))\int_{\X\times \X} d\wedge 1\,\d\tilde{\aalpha}_n\to 0.
		\end{equation}
		Now \eqref{eq: conv mis} can be bounded as follows:
		\begin{align}
			\left|\int_{\X_n\times \X} \left[\Phi(T(y))-\Phi(T_n(x))\right]\varphi(y)\,\d\aalpha_n\right|&\geq \left|\int_{\{\sfd_\Z(T, z)\leq\frac{\varepsilon}{4}\}} \left[\Phi(T(y))-\Phi(T_n(x))\right]\,\d\aalpha_n\right|\\
			&\qquad-2\|\varphi-\mathbbm{1}_{\{\sfd_\Z(T, z)<\frac{\varepsilon}{4}\}}\|_{\L^1(\mm)}\\
			&\geq \int_{E_n} \left( 1-\Phi(T_n)\right)\,\d\aalpha_n -\frac{\delta}{2}\\
			&\geq \frac{\delta}{2}.
		\end{align} 
		It follows that, by choosing the appropriate $\varphi$, the convergence does not hold, which contradicts the hypothesis.
	\end{proof}
	The same statement of Proposition \ref{prop: L0} can be restated in term of something closer to a real distance between maps with different domains:
	\begin{deff}
		Let $(\Z, \sfd_\Z)$ be a complete and separable metric space and $(\aalpha_n)_{n\geq 1}$ be good couplings. For $T_n\in \L^0(\X_n; \Z), T\in\L^0(\X;\Z)$ we write
		\begin{align}
			\sfd_{\L^0}^{\saalpha_n}(T_n, T)&\coloneqq \inf\{\varepsilon>0:\, \aalpha_n\left(\{\sfd_\Z(T_n\circ\proj_1, T\circ\proj_2)>\varepsilon\}<\varepsilon\right)\}\\
			&=\sfd_{\L^0(\saalpha_n)}(T_n\circ\proj_1, T\circ\proj_2).
		\end{align}
	\end{deff}
	\begin{cor}
		Let $(\Z, \sfd_\Z)$ be a complete and separable metric space. Let $T_n: \X_n\to \Z$ be Borel maps, $n\geq 1$, $(\aalpha_n)_n$ good couplings and $T:\X\to \Z$ a Borel map. Then $T_n\xrightarrow{\L^0}{}T$ iff
		\begin{equation}
			\sfd_{\L^0}^{\saalpha_n}(T_n, T)\to 0.
		\end{equation}
	\end{cor}
	\begin{rmk}
		A sort of triangular inequality is in place for this notion of distance. Indeed, for all $S_n, T_n\in \L^0(\X_n; \Z)$, $S, T\in \L^0(\X; \Z)$ it holds
		\begin{equation}
			\sfd_{\L^0}^{\saalpha_n}(S_n, S)\leq \sfd_{\L^0({\scriptstyle \mm_n})}(S_n, T_n)+\sfd_{\L^0}^{\saalpha_n}(T_n, T)+ \sfd_{\L^0({\scriptstyle \mm})}(T, S).
		\end{equation}
	\end{rmk}
	
%	\begin{rmk}
%		It is not difficult to see that the same definition and characterization can be given for maps $f_n:\X_n\to \Y$ where $(\Y, d_\Y)$ is a given complete and separable metric space \add{(Sembra necessario per la uniforme continuità di $\Phi$...)}: the only differences are that one needs to consider $Phi\in \mathcal C^0_b(\Y)$ and the term that must vanish in order to have $\L^0$ convergence is $\aalpha_n(d_y(f_n, f))$.
%	\end{rmk}
	With this notion, we can loosely interpret convergence in concentration as the stability and compactness with respect to Hausdorff distance of 1-Lipschitz functions, where the metric underlying Hausdorff convergence is substituted by the use of good couplings. In particular we have this useful approximation and compactness results:
	\begin{cor}\label{prop: approssimazione e comp L0}
		Suppose $\X_n$ converge in concentration to $\X$ and $(p_n)_{n\geq 1}$ are as in Proposition \ref{prop: caratterizzazione}. Then the following two hold:
		\begin{enumerate}[(i)]
			\item  given $f\in \Lip_1(\X)$ there exist $f_n\in \Lip_1(\X_n)$ such that $f_n\to f$ in $\L^0$;
			\item given $f_n\in \Lip_1(\X_n)$ for $n\geq 1$ with $0\leq f_n\leq 1$, there exists $f\in \Lip_1(\X)$ with $0\leq f\leq 1$ such that, up to subsequences, $f_n\to f$ in $\L^0$.
		\end{enumerate}
	\end{cor}
	\begin{proof}
		Let  $(\aalpha_n)_n$ be good couplings. If $f\in\Lip_1(\X)$, by definition of convergence in concentration, there exist $f_n\in\Lip_1(\X_n)$ such that $\sfd_{\L^0}(f_n, p_n^*f)\to 0$. Now, 
		\begin{align}
			\aalpha_n\left(\{|f_n(x')-f(y)|>\varepsilon\}\right)&\leq \left(\text{proj}_1\right)_{\#}\aalpha_n\left(\{|f_n-f\circ p_n|>\frac{\varepsilon}{2}\}\right)+\\
			&\qquad\aalpha_n\left(\{|f\circ p_n(x')-f(y)|>\frac{\varepsilon}{2}\}\right)\\
			&=\mm_n\left(\{|f_n-f\circ p_n|>\frac{\varepsilon}{2}\}\right)\\
			&\qquad+\aalpha_n\left(\{|f\circ p_n(x')-f(y)|>\frac{\varepsilon}{2}\}\right).
		\end{align}
		The first term converges to 0 by the vanishing of the $\L^0$ distance; the second can be treated as in the proof of Proposition \ref{prop: L0}. By Proposition \ref{prop: L0}, it follows that $f_n\to f$ in $\L^0$.

		Let now $(f_n)_{n\geq 1}$ be as in (ii). By Proposition \ref{prop: caratterizzazione}, there exists a sequence of 1-Lipschitz maps $g_n\in \Lip_1(\X)$ such that $\sfd^{\saalpha_n}_{\L^0}(f_n, g_n)\to 0$, which can be assumed without loss of generality to satisfy $0\leq g_n\leq 1$. By the Ascoli--Arzelà theorem, up to subsequences $g_n\to g$ in $\L^0$ for some $g\in \Lip_1(\X)$. Reasoning as before, it follows that, up to subsequences, $f_n\to f$ in $\L^0$.
	\end{proof}
	% \begin{lm}\label{lemma: comp L0}
	% 	Let $\X_n$ be a sequence of n.m.m.s. spaces converging in concentration to $\X$ and let $p_n:\X_n\to \X$ be a sequence of projection maps inducing the convergence. Let $f_n\in \Lip_1(\X_n)$ for $n\geq 1$ with $0\leq f_n\leq 1$. Then there exists $f\in \Lip_1(\X)$ with $0\leq f\leq 1$ such that, up to subsequences, $f_n\to f$ in $\L^0$.
	% \end{lm}
	% \begin{proof}
	% 	Let $(\aalpha_n)_{n\geq 1}$ be good couplings.
	% 	By Proposition \ref{prop: caratterizzazione}, there exists a sequence of 1-Lipschitz maps $g_n\in \Lip_1(\X)$ such that $\sfd^{\saalpha_n}_{\L^0}(f_n, g_n)\to 0$, which can be assumed without loss of generality to satisfy $0\leq g_n\leq 1$. By the Ascoli--Arzelà theorem, up to subsequences $g_n\to g$ in $\L^0$ for some $g\in \Lip_1(\X)$. Reasoning as in Corollary \ref{cor: approssimazione L0}, it follows that, up to subsequences, $f_n\to f$ in $\L^0$.
	% \end{proof}
	Finally, this characterization of convergence in measure easily allows to recover most of the properties that are known when the domain is fixed.
	\begin{cor} \label{cor: cont}
		Let $f^j_n$ be Borel functions in $\X_n$, $n\geq 1, j=1, \ldots, k$, and $f^j$ Borel functions on $\X$, $j=1, \ldots, k$ be such that $f_n^j\xrightarrow{\L^0}{}f^j$. Then
		\begin{equation}
			G(f^1_n, \ldots, f^k_n)\xrightarrow{\L^0}{}G(f^1, \ldots, f^k)
		\end{equation}
		for all continuous $G$.
	\end{cor}
	We conclude our presentation of $\L^0$ convergence with a couple results related to Young measures.
	\begin{prop}\label{prop: L^0 cs}
		Let $\Z$ be a complete and separable metric space, $T_n:\X_n\to \Z$ be Borel maps, $n\geq 1$, and $T:\X\to \Z$ a Borel map. Then $T_n\to T$ in $\L^0$ if and only if
		\begin{equation}\label{eq: L0 young}
			\left(p_n, T_n\right)_\#\mm_n\rightharpoonup (id, T)_\#m.
		\end{equation}
	\end{prop}
	\begin{proof}
		If \eqref{eq: L0 young} holds, then the convergence follows by testing with functions of the form $\Psi(x, z)=\varphi(x)\phi(z)$ with $\varphi\in C^0_\b(\X)$ and $\phi\in C^0_\b(\Z)$.
		The viceversa follows by noticing that Proposition \ref{prop: L0} implies that the measures $\left(p_n, T_n\right)_\#\mm_n$ are tight and observing that linear combinations of functions $\Psi$ as above are dense in the uniform norm in $C^0_\b(\X\times\Z)$.
	\end{proof}
	\begin{lm}\label{lemma: Young}
		Let $1 <p<\infty$, $f_n\in \L^0(\mm_n)$, $f\in \L^0(\mm)$ be an equi--integrable sequence such that $f_n \mm_n\rightharpoonup f\mm$. Suppose moreover that there exists an interval $J$ such that $f_n(\X_n)\subset J$ and a function $\Theta:J\to \mathbb R$ lsc., strictly convex and bounded from below such that $\int \Theta(f_n)\,\d\mm_n\to\int\Theta(f)\,\d\mm$. Then $f_n\to f$ in $\L^0$.
	\end{lm}
	\begin{proof}
		Let us start by noticing that the family $\mu_n\coloneqq \left(p_n, f_n\right)_\#\mm_n$ is equi--integrable, so, up to subsequences, $\mu_n\to \mu$ for some probability measure $\mu$ on $\X\times\mathbb R$. By Proposition \ref{prop: L^0 cs} it is thus sufficient to prove that $\mu=\left(id, f\right)_\#\mm$. Let us denote by $\mu_x$ the measures obtained by disintegrating $\mu$ with respect to the projection on the first variable, i.e. $\mu=\int \mu_x\,\d\mu$. Fix any $\Phi\in C^0_\b(\X)$: by the equiintegrability and the linear growth of the integrand, it follows
		\begin{align}
			\int \Phi(x)z\, d\mu(x, z)&=\lim_n \int \Phi(x)z\, d\mu_n(x, z)\\
			&=\lim_n \int (\Phi\circ p_n) f_n\,\d\mm_n\\
			&=\int \Phi f\,\d\mm,
		\end{align}
		so, in particular, $\int z\,\d\mu_x=f(x)$ for $m$--a.e. $x$.
		By Jensen inequality, then
		\begin{align}
			\int \Theta(f)\,\d\mm&\leq \int \int\Theta(z)\,\d\mu_x\d\mm\\
			&=\int \Theta(z)\,\d\mu\leq \mathop{\operatorname{liminf}}_n \int \Theta (z)\,\d\mu_n\\
			&=\int \Theta(f_n)\,\d\mm_n.
		\end{align}
		By the hypothesis, the integrals converge and so the inequalities are all equalities: it follows that $\mu_x=\delta_{f(x)}$ for $\mm$--a.e. $x$, which is equivalent to our thesis.
	\end{proof}
	The last form of convergence that we study is $\L^p$ convergence.
	\begin{deff}[Weak and strong $\L^p$ convergence]
		Let $1 <p<\infty$ and $f_n\in \L^p(\mm_n)$, $f\in \L^p(\mm)$. We say that $f_n\xrightharpoonup{\L^p}f$ if $f_n \mm_n\rightharpoonup f \mm$ and $\|f_n\|_{\L^p(\mm_n)}$ is equibounded.
		We say that $f_n\xrightarrow{\L^p}f$ if $f_n\xrightharpoonup{\L^p}f$ and $\|f_n\|_{\L^p(\mm_n)}\to \|f\|_{\L^p(\mm)}$.
	\end{deff}
	\begin{rmk}
		Notice that by convexity and positivity of $|\cdot|^p$, weak convergence implies the lower semicontinuity of the norm. 
	\end{rmk}
	Mimicking what happens when all the functions are defined on the same space, we can generalize also the definitions of convergence in $\L^1$.
	\begin{deff}
		Let $f_n\in \L^1(\mm_n)$, $f\in \L^1(\mm)$. We say that $f_n\xrightharpoonup{\L^1}f$ if $f_n \mm_n$ is equi--integrable and $f_n \mm_n\rightharpoonup f \mm$. We say that $f_n\xrightarrow{\L^1}f$ if $f_n\to f$ in $\L^0$ and $\|f_n\|_{\L^1(\mm_n)}\to \|f\|_{\L^1(\mm_n)}$.
	\end{deff}
	\begin{rmk}
		By a Scheffé--type argument, it can be verified that $f_n\to f$ in $\L^1$ implies $f_n\rightharpoonup f$ in $\L^1$.
	\end{rmk}
	A final useful notion to introduce is $p$--equi--integrability.
	\begin{deff}[$p$--equi--integrability]
		Let $1 <p<\infty$ and $f_n\in \L^p(\mm_n)$. We say that the sequence $(f_n)_{n\geq 1}$ is $p$--equi--integrable if $|f_n|^p\mm_n$ is equi--integrable according to Definition \ref{def: equi--int}.
	\end{deff}
	\begin{rmk}\label{oss: p equi}
		Notice that if $\sup_n\|f_n\|_{\L^p}$ is finite, then $(f_n)$ is $p$--equi--integrable iff 
		\begin{equation}
			\lim_{M\to \infty} \sup_n \int_{\{|f_n|\geq M\}} |f_n|^p\, \d\mm_n=0,
		\end{equation}
		while in general the $p$--equiintegrability hypothesis is weaker.
	\end{rmk}
	We can finally characterize $\L^p$ convergence in terms of properties that are easier to handle.
	\begin{prop}\label{prop: convergenze}
		Let $1 \leq p<\infty$, $f\in \L^p(\mm)$ and $(\aalpha_n)_n$ good couplings. Then the following are equivalent:
		\begin{enumerate}
			\item \label{it: L^p} $f_n\to f$ in $\L^p$;
			\item \label{it: conv dom} $f_n\to f$ in $\L^0$ and $(f_n)_n$ is $p$--equi--integrable; 
			\item \label{it: norma} $\int |f_n(x)-f(y)|^p\,\d\aalpha_n(x, y)\to 0$.
		\end{enumerate}
	\end{prop}
	\begin{proof}
		We show the equivalence in the case $p>1$. If $p=1$ the arguments are similar.
		Assume first that $f_n \to f$ in $\L^0$ and $(f_n)_n$ is $p$--equi--integrable. Let $\varepsilon>0$ and $E_n\coloneqq\{|f_n(x)-f(y)|>\varepsilon\}$. Since $f_n\to f$in $\L^0$, it follows that $\aalpha_n(E_n)\to 0$. We then have
		\begin{align}
			\lims_n \int |f_n(x)-f(y)|^p\,\d\aalpha_n(x, y)&\leq \lims_n\int_{\X_n\times \X\setminus E_n} |f_n(x)-f(y)|^p\,\d\aalpha_n(x, y)\\
			&\quad +2^{p-1}\lims_n\int_{E_n}(|f_n(x)|^p+|f(y)|^p)\,\d\aalpha_n\\
			&\leq \varepsilon^p
		\end{align}
		by $p$--equi--integrability and the fact that $f\in \L^p$.
		By generality of $\varepsilon$, $\|f_n\circ\operatorname{proj}_1-f\circ\operatorname{proj}_2\|_{\L^p(\aalpha_n)}\to 0$.
		
		Assume now that this latter condition holds. By the triangle inequality, it is immediate that $\|f_n\|_{\L^p}\to \|f\|_{\L^p}$. Reasoning as in Proposition \ref{prop: L0}, it follows that $f_n\to f$ in $\L^p$.
		
		Finally, assume that $f_n\to f$ in $\L^p$. By Lemma \ref{lemma: Young}, it follows that $f_n\to f$ in $\L^0$. In particular, by Corollary \ref{cor: cont} it follows that, for each $M>0$, $|f_n|\wedge M\to |f|\wedge M$ in $\L^p$. By difference and recalling Remark \ref{oss: p equi}, it is easy to check that $(f_n)_n$ is $p$--equi--integrable.
	\end{proof}
	\begin{rmk}
		With these results at hand, it is easy to see that a further characterization of $\L^0$ convergence is possible: fixed $p>1$, a sequence $f_n\in \L^0(\X_n), n\geq 1$ converges in $\L^0$ to $f\in \L^0(\X)$ iff $-M\vee f_n\wedge M$ converges in $\L^p$ to $-M\vee f\wedge M$ for each $M>0$.
	\end{rmk}
	\begin{cor}\label{cor: coupling}
		Let $1<p, p'<+\infty$ be such that $\frac{1}{p}+\frac{1}{p'}=1$. Assume that $f_n\in \L^p(\mm_n)$, $g_n\in \L^{p'}(\mm_n)$ are such that $f_n\to f$ in $\L^p$ and $g_n\rightharpoonup g$ in $\L^{p'}$. Then $f_ng_n\rightharpoonup fg$ in $\L^1$.
	\end{cor}
	\begin{proof}
		Let $E_n\in \mathcal{B}(\X_n)$ be such that $\mm_n(E_n)\to 0$. It follows
		\begin{equation}
			(|f_n g_n|\,\mm_n)(E_n)\leq\sup_n\|g_n\|_{\L^{p'}}\left(\int_{E_n}|f_n|^p\,\d\mm_n\right)^{\frac{1}{p}}\to 0,
		\end{equation}
		so the equi--integrability follows. Let now be $\varphi\in C^0_\b(\X)$. Let $\varepsilon>0$ and $\psi\in \operatorname{Lip}(\X)$ such that $\|f-\psi\|_{\L^p}<\varepsilon$.
		We can estimate
		\begin{align}
			\left| \int\varphi\circ p_n f_ng_n\,\d\mm_n-\int\varphi\psi g_n\,\d\mm_n\right|&\leq \|\varphi\|_\infty \sup_n \|g_n\|_{\L^{p'}}\left(\int|f_n-\psi\circ p_n|^p\,\d\mm_n\right)^\frac{1}{p}.
		\end{align}
		It is enough to show that the term in the parenthesis can be bounded in terms of $\varepsilon$. Indeed, if $\aalpha_n$ are good couplings,
		\begin{align}
			\int|f_n-\psi\circ p_n|^p\,\d\mm_n&\leq 2^{p-1}\bigg(\|f_n\circ\operatorname{proj}_1-\psi\circ\operatorname{proj}_2\|^p_{\L^p(\aalpha_n)}\\
			&\qquad\qquad\quad +\int |\psi(x)-\psi(y)|\,\d(id, p_n)_{\#}\aalpha_n\bigg)\\
			&\leq \varepsilon^p+o(1),
		\end{align}
		where the last term can be shown to be vanishing as in the proof of Proposition \ref{prop: L0}.
	\end{proof}
	\begin{cor}
		Let $f_n, g_n \in \L^p(\X)$, $1\leq p<\infty$ and suppose that $f_n\to f$ in $\L^p$. Then $g_n\to f$ in $\L^p$ iff $\|f_n-g_n\|_{\L^p}\to 0$.
	\end{cor}
	A further characterization of $\L^0$ convergence can be given in terms of pushforwards of weakly convergent sequence of measures:
	\begin{prop}\label{prop: L0 via misure}
		Let $(\Z, \sfd_\Z)$ be a complete and separable metric space. Let $T_n: \X_n\to \Z$ be Borel maps, $n\geq 1$, and $T:\X\to \Z$ a Borel map. Then the following are equivalent:
	\begin{itemize}		
	\item[i)] $T_n\xrightarrow{\L^0}{}T$
	\item[ii)] For every $C>1$ and every sequence of probability measures $\mu_n\in \pr(\X_n)$ with $\mu_n\leq C\mm_n$ weakly converging to some $\mu\in\pr(\X)$ we have $(T_n)_\#\mu_n\weakto T_\#\mu$
	\item[iii)] For every equi-integrable sequence of probability measures $\mu_n\in \pr(\X_n)$   weakly converging to some $\mu\in\pr(\X)$, we have $(T_n)_\#\mu_n\weakto T_\#\mu$
	\end{itemize}		
	\end{prop}
	\begin{proof}
		We prove that (iii) implies (i), then that (i) implies (ii) and conclude by showing that (ii) implies (iii). 
		\begin{itemize}
			\item Assume item (iii) holds. Let $\Phi\in C^0_\b(\Z)$ and let $\varphi\in C^0_\b(\X)$. Then 
			\begin{equation}
				\int_\X \varphi\, \d (p_n)_\#(\Phi\circ T_n)\mm_n=\int_{\X_n} (\varphi\circ p_n)(\Phi\circ T_n)\,\d\mm_n=\int_\Z \Phi\,\d(T_n)_\#(\varphi\circ p_n\, \mm_n).
			\end{equation}
			Set $\mu_n\coloneqq \left(\varphi\circ p_n\right)\, \mm_n$ and notice that $\left(p_n\right)_\# \mu_n=\varphi \left(p_n\right)_\# \mm_n$, so $\mu_n$ weakly converges to $\mu\coloneqq \varphi\, \mm$. This implies (i).
			\item Assume (i) holds. Let $\mu_n= \rho_n\, \mm_n$, $\mu=\rho\, \mm$ and notice that $\rho_n\rightharpoonup \rho$ in $\L^p$ for all $p>1$. Let $\Phi\in C^0_\b(\Z)$. It holds 
			\begin{equation}
				\int \Phi\,\d(T_n)_\#\mu_n=\int \left(\Phi\circ T_n\right)\rho_n\,\d\mm_n.
			\end{equation}
			Now $\Phi\circ T_n\to \Phi\circ T$ in $\L^p$ for all $p>1$, so by Corollary \ref{cor: coupling} item (ii) follows.
			\item Assume (ii) holds. Item (iii) follows by truncation. Indeed, if we set \begin{equation}
				\mu_{n, M}\coloneqq \frac{1}{(\mu_n\wedge M)(\X_n)} (\mu_n\wedge M),
			\end{equation}
			we have that $\|\mu_{n, M}-\mu_n\|_{TV}\xrightarrow[m\to \infty]{} 0$ uniformly in $n$ by equi--integrability. Up to subsequences, now, $\mu_{n, M}\xrightharpoonup[n\to \infty]{}\tilde{\mu}_M$ with $\|\tilde{\mu}_{M}-\mu\|_{TV}\xrightarrow[m\to \infty]{} 0$. Approximating $\mu_n$ with $\mu_{n, M}$ and applying (ii) yields (iii).
		\end{itemize}
	\end{proof}	
We close this section with the definition of a last notion of convergence.
\begin{deff}\label{def: L0 mappe}
	Let $(\Y_n, \sfd^\Y_n, \mm^\Y_n)$ be a sequence of n.m.m.s. and let $p^\Y_n:\Y_n\to \Y$ be a sequence of Borel projections to a n.m.m.s. $(\Y, \sfd^\Y, \mm^Y)$ such that $(p^\Y_n)_\#\mm^\Y_n\weakto \mm^\Y$. 
	Let $T_n:\X_n\to \Y_n$ be a sequence of Borel maps such that $(T_n)_\#\mm_n$ is equi--integrable. 
	We say that $T_n$ converges in measure to $T:\X\to \Y$, or $T_n\xrightarrow{\L^0}T$, if
	\begin{equation}
		p^\Y_n\circ T_n \xrightarrow{\L^0} T.
	\end{equation}
\end{deff}
% \begin{rmk}
% 	Notice that, whenever $\Y_n\equiv\Z$ and the maps $T_n$ satisfy this weak hypothesis on the compression, this notion of convergence is obviuosly a generalization of the one in Definition \ref{def: conv L0}. Without this hypothesis is clear that one cannot hope to find an intrinsic notion of convergence when the spaces are only converging in concentration
% \end{rmk}
Notice that many of the results and characterization proved in this section for $\L^0$ convergence when the target is a fixed separable metric space (without measure) can easily be extended to this more general setting. In particular, we state the following generalization of Proposition \ref{prop: L0 via misure}, whose proof is quite immediate.
\begin{prop}\label{prop: L0 via misure bis}
	Let $\Y_n, \Y$ as in Definition \ref{def: L0 mappe} and let $T_n:\X_n \to \Y_n$. Then the following are equivalent:
	\begin{enumerate}[i)]
		\item $T_n\xrightarrow{\L^0}T$;
		\item for every $C>1$ and every sequence of probability measures $\mu_n\in \pr(\X_n)$ with $\mu_n\leq C\mm_n$ weakly converging to some $\mu\in\pr(\X)$, we have that $(T_n)_\#\mu_n$ is equi--integrable and converges weakly to $T_\#\mu$;
		\item for every equi--integrable sequence of probability measures $\mu_n\in \pr(\X_n)$ weakly converging to some $\mu\in\pr(\X)$, we have that $(T_n)_\#\mu_n$ is equi--integrable and converges weakly to $T_\#\mu$.
	\end{enumerate}
\end{prop}
\begin{rmk}
	The hypothesis on the compression allows to get a general result in the spirit of Corollary \ref{cor: cont}. Indeed it is easy to see that the composition of sequences of maps converging in the sense of Definition \ref{def: L0 mappe} converges to the composition of the limits. In the same way, given a sequence $T_n:\X_n\to \Y_n$ converging in measure to $T:\X\to \Y$ and a sequence of maps to some separable metric space $G_n:\Y_n\to \Z$ which converges to $G:\Y\to \Z$ in the sense of Definition \ref{def: conv L0}, then $G_n\circ T_n$ converges to $G\circ T$ in the sense of Definition \ref{def: conv L0}.
\end{rmk}

\section{Considerations about limits}\label{section: technical}
In this section we prove two important results about the structure of limits of measures (as well as functions) living on sequences of n.m.m.s. converging in concentration. In the first subsection we address the issue of the, in some sense, ``geometricity'' of limits, while in the second we prove an important technical procedure that will be instrumental to prove the stability of the slope of the entropy in Subsection \ref{subsection: stab}. 

\subsection{On the uniqueness of limits}
In Section \ref{section: conv} we have recalled how, given a sequence of spaces $\X_n$ converging in concentration to a limit space $\X$, an equi--integrable sequence of probability measures $\mu_n$ and a sequence of projections $p_n:\X_n\to\X$ in the sense of Proposition \ref{prop: caratterizzazione}, one can give meaning to a weak limit of the sequence. Anyhow, this limit might depend a priori on the choice of the sequence of projections, which is not canonical. In this subsection we show how, in reality, the limit is unique, up to the action of the group of automorphisms of the limit space. 
\begin{rmk}
	In this subsection we will only uniqueness results about the convergence of equi--integrable sequences of probability measures. Since all of the definitions of convergences given in Section \ref{section: conv} are directly based on - or characterized through - this notion, it is easy to convince oneself that this really yields a uniqueness result for general limits.
\end{rmk}

The precise statement is the following:
\begin{thm}\label{teo: unicità limite}
	Let $\X_n$ be a sequence of n.m.m.s. spaces converging in concentration to both the (isomorphic) spaces $\X^1$ and $\X^2$ and let $p_n^1:\X_n\to \X^1$, $p_n^2:\X_n\to \X^2$ be two sequences of projection maps as in Proposition \ref{prop: caratterizzazione} inducing the convergence. 
	Then, for all measure preserving isometries $\Phi:\X^1\to\X^2$,  there exists a sequence of measure preserving isometries $\Phi_n:\X^1\to\X^1$ such that, given any equi--integrable sequence $\mu_n\in \pr (\X_{n})$ and any subsequence $k\mapsto n_k$, $\mu_{n_k}$ converges weakly to $\mu^{1}$ with respect to $p_{n_k}^1$ if and only if it converges weakly to $\mu^2=\Phi_\#\mu^1$ with respect to $\Phi_n\circ p_{n_k}^2$.
\end{thm}

\begin{rmk}\label{oss: isom}
Notice that that the uniqueness of the concentration limit already tells that two possible limits of the same sequence of measure are isomorphic in a way that may depend on the sequence itself:
Assume $\X_n\xrightarrow{conc} \X$ and the convergence is realized (in the sense of Proposition \ref{prop: caratterizzazione}) by two sequences of projections $(p_n)_{n\geq 1}$ and $(q_n)_{n\geq 1}$. Then whenever $(\mu_n)_{n\geq 1}$ is an equi--integrable sequence converging to $\mu_p$ with respect to the structure induced by $(p_n)_{n\geq 1}$ and to $\mu_q$ with respect to $(q_n)_{n\geq 1}$, it is easy to see that there exists an isometry $\phi: \supp \mu_p\to \supp \mu_q$ such that $\phi_{\#} \mu_p=\mu_q$. Indeed equi--integrability implies that $(\X_n, \sfd_n, \mu_n)$ converges in concentration to both $(\X, \sfd, \mu_p)$ and $(\X, \sfd, \mu_q)$, which must then be isomorphic.
Notice however that the isomorphism $\phi$ depends a priori on the sequence $(\mu_n)_{n\geq 1}$. The added value of Theorem \ref{teo: unicità limite} is that it shows that the identification between limit measures is actually, in some sense, canonical.
\end{rmk}

Building on Section \ref{section: conv}, we can link convergence properties of the projections and the relations between limits with respect to them. Indeed Proposition \ref{prop: L0 via misure} has this interesting Corollary:
\begin{cor}
	Let $\X_n$ be a sequence of n.m.m.s. spaces converging in concentration to both the (isomorphic) spaces $\X^1$ and $\X^2$ and let $p_n^1:\X_n\to \X^1$, $p_n^2:\X_n\to \X^2$ be two sequences of projection maps as in Proposition \ref{prop: caratterizzazione} inducing the convergence. Let $\Phi:\X^1\to\X^2$ be a Borel bijection. Then the following two statements are equivalent:
	\begin{itemize}
		\item[i)] Given any equi--integrable sequence $\mu_{n}\in \pr (\X_{n})$, it converges weakly to $\mu^{1}$ with respect to $p_{n}^1$ if and only if it converges weakly to $\mu^2=\Phi_\#\mu^1$ with respect to $p_{n}^2$;
		\item[ii)] It holds $p_n^2\xrightarrow[p_n^1]{\L^0} \Phi$.
	\end{itemize}
\end{cor}
\begin{rmk}
	Item (ii) is to be compared with the immediate $p_n\xrightarrow[p_n]{\L^0} {\rm id}$ holding for all projections $p_n$.
\end{rmk}
\begin{proof}
	Statement (ii) follows immediately from (i) by Proposition \ref{prop: L0 via misure}. Viceversa, assume that (ii) holds and let $\mu_n$ be as in the statement. Let $\mu^1\in \pr(\X^1)$ be a weak limit point of $\mu_n$. By equi--integrability, there exist a subsequence $\mu_{n_j}$ such that 
	\begin{align}
		&\mu_{n_j}\xrightharpoonup[p^1_{n_j}]{}\mu^1\\
		&\mu_{n_j}\xrightharpoonup[p^2_{n_j}]{}\mu^2
	\end{align}
	for some $\mu^2\in\pr(\X^2)$.
	By (ii) and Proposition \ref{prop: L0 via misure}, then, $\mu^2=\Phi_\# \mu^1$. In the same way, if $\mu^2$ is a weak limit point of $\mu_n$ with respect to $p^2_n$, then  $\mu^2=\Phi_\# \mu^1$ for some $\mu^1\in\pr(\X^1)$. Since $\Phi_\#$ is a bijection, the thesis follows.
\end{proof}
\begin{rmk}\label{oss: conseguenza equivalenza}
	Recalling the definition of equivalent projections in Remark \ref{oss: equivalenza} and reasoning as above, we see that the following three statements are equivalent:
	\begin{itemize}
		\item two sequences of projections $p^{1, 2}_n:\X_n\to \X$ as in Proposition \ref{prop: caratterizzazione} are equivalent;
		\item it holds $p^{1, 2}_n\xrightarrow[p^{2,1}_n]{\L^0}id_{\X}$;
		\item given any equi--integrable sequence of probability measures $n\mapsto\mu_n\in\pr(\X_n)$ and $\mu\in\pr(\X)$, we have that $\mu_n\xrightharpoonup[p^1_n]{}\mu$ if and only if $\mu_n\xrightharpoonup[p^2_n]{}\mu$.
	\end{itemize}
\end{rmk}
The proof of Theorem \ref{teo: unicità limite} rests on a compactness argument which, in turns, lies on this $\L^0$ compactness property of projections inducing the convergence.
\begin{thm}\label{teo: comp proiezioni}
	Let $\X_n$ be a sequence of n.m.m.s. spaces converging in concentration to both $\X^1$ and $\X^2$ and let $p_n^1:\X_n\to \X^1$, $p_n^2:\X_n\to \X^2$ be two sequences of projection maps as in Proposition \ref{prop: caratterizzazione} inducing the convergence. 
	Then $\X^1$ is isomorphic to $\X^2$ and, up to subsequences, 
	\begin{equation}
		p_n^2\xrightarrow[p_n^1]{\L^0} \Phi,
	\end{equation}
	where $\Phi$ is a measure preserving isometry.
	% Then there exists a subsequence $j\mapsto n_j$ and a measure preserving isometry $\Phi:\X^1\to \X^2$ such that, given an equi--integrable sequence $\mu_{n_j}\in \pr (\X_{n_j})$, it converges weakly to $\mu^{1}$ with respect to $p_{n_j}^1$ if and only if it converges weakly to $\mu^2=\Phi_\#\mu^1$ with respect to $p_{n_j}^2$.
\end{thm}
\begin{rmk}
	We emphasided the fact that $\X^1$ and $\X^2$ are isomorphic, which is already known by abstract means, because this result could be used to prove the uniqueness of limit when this was defined directly through the sequence of projections of Proposition \ref{prop: caratterizzazione}.
\end{rmk}
In order to prove this result, let us recall some definitions and lemmata.
\begin{deff}[see \cite{Sh16}]\label{def: 1-Lipschitz err}
	Let $\X$ be a n.m.m.s. and $\Z$ be a metric space. We say that a measurable map $f:\X\to \Z$ is 1-Lipschitz up to an additive error $\varepsilon>0$ if the following holds: there exists $E\subset \X$ with $\mm_\X(E)\leq \varepsilon$ such that 
	\begin{equation}
		\sfd_\Z(f(x), f(y))\leq \sfd_\X(x, y)+\varepsilon
	\end{equation}
	for $x, y$ in $\X\setminus E$.
\end{deff}
\begin{lm}
	Let $\X$ be a n.m.m.s. and $\Z$ be a metric space. Let $f:\X\to \Z$ be a 1-Lipschitz map and let $\varepsilon>0$. If $g:\X\to \Z$ is a measurable map such that $\sfd_{\L^0}(f, g)\leq \varepsilon$, then $g$ is 1-Lipschitz up to an additive error $2\varepsilon$.
\end{lm}

The following Lemma is a minor variation on a result in \cite{Sh16}.
\begin{lm}[{\cite[Lemma 5.4]{Sh16}}]\label{lemma: Lipschitz appr}
	Let $\Y$ be a nmms. and let $f:\Y\to \ell^\infty$ be 1-Lipschitz up to an additive error $\varepsilon>0$ in the sense of \cite{Sh16}, i.e. there exists $\tilde{\Y}\subset \Y$ with $\mm_\Y(\tilde{\Y})\geq 1-\varepsilon$ such that $\|f(x)-f(y)\|_{\ell^\infty}\leq \sfd_\Y(x, y)+\varepsilon$ for all $x, y\in \tilde{\Y}$. Then there exists a 1-Lipschitz map $g:\Y\to \ell^\infty$ such that $\sfd_{\L^0}(f, g)\leq \varepsilon$, where $\sfd_{\L^0}$ is defined as in \eqref{eq: dist L0}.
\end{lm}
\begin{proof}
	The proof is the same as in \cite{Sh16}, noticing that the same McShane--type argument can be carried on on the infinitely many components of $f:\Y\to \ell^\infty$.
\end{proof}
The proof of Theorem \ref{teo: comp proiezioni} is based on three main facts: the fact that Lipschitz maps from $X_n$ to $\ell^\infty$ con be approximated by Lipschitz maps with values in a finite dimensional space and an Ascoli--Arzelà type compactness argument.
\begin{lm}\label{lemma: approssimazione dim infinita}
	Let $\X_n$ be a sequence of n.m.m.s.. Let $f_n:\X_n\to \ell^\infty$ be a sequence of measurable maps such that $(f_n)_\# \mm_n$ is a tight sequence. Then for all $\varepsilon>0$ there exists $N\geq 1$ and a sequence $g_n:\X_n\to \ell^\infty$ such that the following hold:
	\begin{itemize}
		\item we have $\sfd_{\L^0}(f_n, g_n)\leq \varepsilon$;
		\item there exists $V\subset \ell^\infty$ a vector space of dimension $N$ containing the images of all the $g_n$'s.
	\end{itemize}
\end{lm}
\begin{proof}
	By tightness, there exist $z_1, \ldots, z_N$ in $\ell^\infty$ such that $(f_n)_\# \mm_{\X_n}(\bigcup_j B\left(z_j, \frac{\varepsilon}{2}\right))\geq 1-\frac{\varepsilon}{2}$ for all $n$. It is now enough to set $V$ as the vector space generated by the $z_j$'s and to define $g_n$ as a projection of $f_n$ on $V$ (cfr. with \cite[Lemma 3.5]{Sh16}).
\end{proof}
\begin{cor}
	Let $\X_n$ be a sequence of n.m.m.s.. Let $f_n:\X_n\to \ell^\infty$ be a sequence of 1-Lipschitz maps such that $(f_n)_\# \mm_n$ is weakly converging. Then for all $\varepsilon>0$ there exists $N\geq 1$ and a sequence of 1-Lipschitz maps $g_n:\X_n\to \ell^\infty$ such that the following hold:
	\begin{itemize}
		\item we have $\sfd_{\L^0}(f_n, g_n)\leq \varepsilon$;
		\item there exists $V\subset \ell^\infty$ a vector space of dimension $N$ containing the images of all the $g_n$'s.
	\end{itemize}
\end{cor}
\begin{proof}
	Since the maps are Lipschitz, the supports of the pushforward measures live in a $\sigma$--compact (thus measurable) subset of $\ell^\infty$: the pushforward measures are thus equi--tight.
	The result follows directly from Lemma \ref{lemma: approssimazione dim infinita}, Definition \ref{def: 1-Lipschitz err} and Lemma \ref{lemma: Lipschitz appr}.
\end{proof}
\begin{lm}
	Let $\X$ be a n.m.m.s. and let $(f_n)_{n\geq 1}$ be a sequence of maps such that $f_n:\X\to \ell^\infty$ is 1-Lipschitz up to an additive error $\varepsilon_n$, with $\varepsilon_n\to 0$. Assume that $(f_n)_\#\mm_n$ is weakly converging. Then $(f_n)_{n\geq 1}$ is precompact in the $\L^0$ distance.
\end{lm}
\begin{proof}
	Let us prove that for all $\varepsilon>0$ there exists a sequence $g^n_\varepsilon$ which is compact and $\varepsilon$ close to $f_n$.
	By Lemma \ref{lemma: Lipschitz appr} we might assume that the $f_n$'s are 1--Lipschitz. Reasoning as before, we might also assume that the pushforward measures are tight, so that we can apply Lemma \ref{lemma: approssimazione dim infinita} and get the existence of a sequence $g_n$ of 1--Lipschitz map with values in a finite dimensional space which are $\varepsilon$--close to $f_n$. The compactness of $g_n$ follows (for instance) from Ascoli--Arzelà.  
\end{proof}
\begin{cor}\label{cor: compattezza dim inf}
	Let $\X_n$ be a sequence of n.m.m.s. such that $\X_n\concto \X$ and let $f_n:\X_n\to \ell^\infty$ be a sequence of 1--Lipschitz maps up to an additive error $\varepsilon_n\to 0$. Assume that $(f_n)_\#\mm_n$ is weakly converging. Then $(f_n)_{n\geq 1}$ is precompact with respect to the $\L^0$ convergence.
\end{cor}
\begin{proof}
	Let $\aalpha_n$ be good couplings.	Reasoning as before, we may also assume that the maps are 1--Lipschitz. Let us prove that for each $\varepsilon>0$ there exists an $\L^0$ compact sequence $g_n:\X\to \ell\infty$ such that $\sfd^\saalpha_{\L^0}(f_n, g_n)\leq 2\varepsilon$. By Lemma \ref{lemma: approssimazione dim infinita} we can find a sequence $h_n$ of 1--Lipschitz maps with values in a finite dimensional space $V$ which are $\varepsilon$--close to $f_n$. Reasoning componentwise as in the proof of Proposition \ref{prop: approssimazione e comp L0}, we can find a sequence $g_n$ of 1--Lipschitz maps with values in $V$ which are $\varepsilon$--close to $h_n$ and thus such that $\sfd^\saalpha_{\L^0}(f_n, g_n)\leq 2\varepsilon$. The result now follows e.g. from the previous Lemma.
\end{proof}
\begin{proof}[Proof of Theorem {\ref{teo: comp proiezioni}}]
	Up to an isometric immersion, we can assume that $\X\subset \ell^\infty$.
	Let $p^1_n, p^2_n$ be projection maps inducing the convergence as in the statement. By Proposition \ref{prop: caratterizzazione} we know that $p^2_n$ are 1--Lipschitz up to an additive error $\varepsilon_n\to 0$ and that $(p^2_n)_\#\mm_n$ is weakly converging. By Corollary \ref{cor: compattezza dim inf} we know that up to subsequences $p^2_n\to \Phi$ in $\L^0$ for some 1--Lipschitz measure preserving map $\Phi:\X^1\to\X^2$. Since the argument can be carried on in the same way for $p^1_n$, we get the existence of $\Psi:\X^2\to\X^1$ 1--Lipschitz and measure preserving. This now implies that $\X^1$ and $\X^2$ are isomorphic and thus that $\Phi$ is an isometry.
\end{proof}

We are finally ready to prove the main result of this subsection:
\begin{proof}[Proof of Theorem \ref{teo: unicità limite}]
	Given $\Phi$ as above, we want to show that there exist $\Phi_n$ such that $\Phi_n\circ p_n^2\xrightarrow[p^1_n]{L^0} \Phi$. Let us fix $\aalpha^1_n$ good couplings with respect to $p^1_n$ and notice that, by Theorem \ref{teo: comp proiezioni} and an easy compactness argument, for $\varepsilon>0$ and $n$ big enough
	\begin{equation}
		\sfd^{\saalpha^1_n}_{\L^0}(p_n^2, {\rm Iso}(\X^1; \X^2))<\varepsilon
	\end{equation}
	holds. In particular, there exist $\Phi_n^\varepsilon$ such that, eventually in $n$,
	\begin{equation}
		\sfd^{\saalpha^1_n}_{\L^0}(\Phi^\varepsilon_n\circ p_n^2, \Phi)<\varepsilon.
	\end{equation}
	By a diagonal argument, the thesis follows.
\end{proof}
Notice that, a posteriori, the precompactness result of Theorem \ref{teo: comp proiezioni} can be reinterpreted as a consequence of a compactness result for the isomorphisms $\Phi_n$. Indeed the following holds:
\begin{lm}
	Let $(\Z, \sfd_Z, \mm_\Z)$ be a n.m.m.s.. Then 
	\begin{equation}
		{\rm Aut}(\Z)\coloneqq\{ \Phi:\Z\to \Z \text{ measure preserving isometry}\}
	\end{equation}
	is compact in the $\L^0$ topology.
\end{lm}
\begin{proof}
	Given $m\geq 1$, by tightness we can choose $K^m\subset \X$ a compact set such that $\mm_\Z(K^m)\geq 1-\frac{1}{m}$. Define $K^m_n\coloneqq T_n^{-1}(K^m)$ and $E^m=\lims_n K^m_n$. Now, $\mm_\Z(E^m)\geq \inf_n \mm_\Z(K^m_n)\geq 1-\frac{1}{m}$; by an Ascoli-Arzelà type argument it is also easy to see that there exists $T^m:\Z\to\Z$ such that, up to subsequences, $T_n\to T^m$ pointwise on $E^m$. By a diagonal argument, we might assume that $T^m=T^n$ on $E^m$ for $n\geq m$. Finally, since $\lim_m\lims_n \sfd_{\L^0}(T_n, T^m)=0$ and $T^m$ is $\L^0$--Cauchy, the thesis follows.
\end{proof}
\subsection{A ``guided pullback" procedure} \label{subsection: pullback guidato} 

	% \begin{manualtheorem}{\ref{teo: pullback guidato}}[Guided pullback]
	% 	This is a theorem.
	% \end{manualtheorem}

	Throughout this subsection we fix a sequence $n\mapsto (\X_n,\sfd_n,\mm_n)$ of normalized metric measure spaces converging in concentration to a limit space $(\X,\sfd,\mm)$. We also fix a corresponding sequence of `projections' $p_n:\X_n\to \X$  as in Proposition \ref{prop: caratterizzazione}.
	
	% {\color{red} As we have seen, in these circumstances we can give a meaning to the concept of weakly converging sequence of measures $\mu_n\weakto \mu$ for $\mu_n$ and, by xxx, if the $\mu_n$'s are equi-integrable, the limit measure $\mu$ only depends on the restriction of the $\mu_n$'s to the subsets $\hat \X_n\subset \X_n$ where the projections $p_n$ are (almost) 1-Lipschitz.}

% It follows that if $\mu_{1,n},\mu_{2,n}\in \pr(\X_n)$ are equi-integrable sequences of measures weakly converging to $\mu_1,\mu_2\in \pr(\X)$ respectively, we have
% \begin{equation}
% \label{eq:w2lsc}
% \W_2(\mu_1,\mu_2)\leq\limi_{n\to\infty} \W_2(\mu_{1,n},\mu_{2,n}).
% \end{equation}
The question we want to address here is the following. Suppose that we are given $(\mu_{1,n}),(\mu_{2,n})$, can we `slightly modify' one of the two sequences in order to have convergence of the $\W_2$-distances? The main result of this section is the following theorem, where `slightly modify' is interpreted in the sense of small difference in the Total Variation norm:\footnote{This approximation result (that can be regarded as an extension of \cite[Lemma 3.15]{FuSh13}, see also Lemma \ref{lm: approssimazione}) is crucial in the proof of Theorem \ref{teo: pullback guidato}. When dealing with the $\Gamma-\limi$ inequality for the slope, we are given a converging sequence of measures $\mu_n\rightharpoonup\mu$ and we want to prove that the slope of the entropies is lower semicontinuous in the limit. In order to approximate the difference quotient appearing in \eqref{eq: sup slope}, given $\nu \in\pr(\X)$, we would like to approximate both the entropy and the Wasserstein distance from $\mu$: Theorem \ref{teo: pullback guidato} will allow us to produce a sort of ``guided'' pullback $\nu_n\in\pr (\X_n)$ of $\nu$ which can be thought as ``parallel'' to $\mu_n$.}
	\begin{thm}\label{teo: pullback guidato}
		Let $(\X_n,\sfd_n,\mm_n)\stackrel{conc}\to (\X,\sfd,\mm)$, $p_n:\X_n\to\X$ be projection maps as in Proposition \ref{prop: caratterizzazione} and let $\mu_{1, n}, \mu_{2, n}\in \pr(\X_n)$ be equi--integrable sequences weakly converging to $\mu_1,\mu_2\in \pr(\X)$ respectively.
		
		Then for every $\eps>0$ there exist $\nu_n\in \pr(\X_n)$, $n\in\N$, such that:
		\begin{subequations}
		\begin{align}
		\label{eq:boundW2}
		 \lims_{n\to\infty}&\W_2(\nu_n, \mu_{1, n})\leq \W_2(\mu_2, \mu_1),\\
		\label{eq:boundTV}
		\lims_{n\to\infty}& \|\nu_n- \mu_{2, n}\|_{\sf TV}=0,\\
		\label{eq:boundnun}
		\nu_n&\leq \mu_{1, n}+(1+\eps )\mu_{2, n},\qquad\forall n\in\N,\text{ sufficiently big}.
		\end{align}
		\end{subequations}
		\end{thm}
%	In general the spaces are not equibounded (in which case the proof can be significantly simplified), so we have to deal with sort of a ``patchwork'' argument. 
Since the proof of Theorem \ref{teo: liminf slope} is quite technical, we have chosen to split it in a few Lemmata, in order to keep the presentation as clear and linear as possible. The core argument is in the following result:
%	 Technicalities a part, the gist of the construction is in the following Lemma, which is a variation on the theme of \cite[Lemma 3.13]{FuSh13}.
	\begin{lm}[Transverse transport]\label{lemma: trasporto} With the same assumptions and notation as in Theorem \ref{teo: pullback guidato} the following holds.
	
	Assume that the limit measures $\mu_1,\mu_2$ are concentrated on the Borel sets $B_1,B_2\subset\X$ with ${\rm diam}(B_{1, 2})\leq  \frac{\delta}{2}$ respectively and for some $0<\delta<\frac18$. Let  ${\sf D}\coloneqq  \operatorname{diam}(B_1\cup B_2)$. 
	
	Then for every $\eps>0$  there exist $\nu_n\in \pr(\X_n)$, $n\in\N$, such that:
		\begin{subequations}
		\begin{align}
		\label{eq:boundW21}
		\lims_{n\to\infty} \W_2(\mu_{1, n}, \nu_n)&\leq (1+\delta^\frac12){\sf D}+5\delta^\frac12,\\
		\label{eq:boundTV1}
		\lims_{n\to\infty} \|\nu_n-\mu_{2, n}\|_{\sf TV}&\leq 3\delta^\frac12 ,\\
		\label{eq:boundnun1}
\text{for every $n$ }\qquad		\nu_n&\leq \mu_{1, n}+(1+\eps)\mu_{2, n}.%\qquad\forall n\in\N.
		\end{align}
		\end{subequations}
	\end{lm}
\begin{proof} If we find $(\nu_n)$ satisfying \eqref{eq:boundnun1} for $n$ sufficiently big and also \eqref{eq:boundW21}, \eqref{eq:boundTV1}, then redefining the sequence for a finite number of terms we achieve the conclusion.

Assume at first that $\mu_{i,n}$ is concentrated on $p_n^{-1}(B_i)\cap K_n$, where $K_n$ is the regular set as in Proposition \ref{prop: caratterizzazione}, for every $n\in\N$, $i=1,2$. Let us fix ${\sf D}'> {\sf D}$ and then put $\sfd_n'=\sfd_n\wedge {\sf D}'$. 
Similarly to the proof of \cite[Lemma 3.13]{FuSh13}, we make the following claim: letting $\W_1'$ be the  1-Wasserstein distance with respect to $ \sfd_n'$, it holds
\begin{equation}
\label{eq:claim: dist wass}
\lims_{n\to\infty}  {W}'_1(\mu_{1, n}, \mu_{2, n})\leq {\sf D}.
\end{equation}
By Kantorovich duality, to prove this it suffices  to show that for any choice of functions $f_n:\X_n\to\R$, $n\in\N$, that are 1-Lipschitz w.r.t.\ $\sfd_n'$ we have
\begin{equation}
\label{eq:claim: dist wassdual}
\lims_{n\to\infty} \int f_n\,\d\mu_{2,n}-\int f_n\,\d\mu_{1,n}\leq {\sf D}.
\end{equation}
Fix such $f_n$'s, notice that they are 1-Lipschitz also w.r.t.\ $\sfd_n$, then recall   Definition \ref{def: convergenza in concentrazione} and Proposition \ref{prop: caratterizzazione}, to get existence of 1-Lipschitz functions $g_n:\X\to\R$  such that for any $\eps>0$ we have $\lims_n\mm_n(\{|f_n-g_n\circ p_n|>\eps\})=0$. By equi-integrability we also get $\lims_n\mu_{i,n}(\{|f_n-g_n\circ p_n|>\eps\})=0$, $i=1,2$. In particular, eventually there are points $x_{n}\in p_n^{-1}(B_1\cup B_2)$ such that $|f_n-g_n\circ p_n|\leq\eps$ and thus eventually in $n$ we have the uniform bound
\[
\begin{split}
\sup_{p_n^{-1}(B_1\cup B_2)}|f_n-g_n\circ p_n|&\leq \eps+\sup_{p_n^{-1}(B_1\cup B_2)}|f_n(\cdot)-f_n(x_{n})|+\sup_{p_n^{-1}(B_1\cup B_2)}|g_n( p_n(\cdot))-g_n( p_n(x_n))|\\
&\leq \eps+{\sf D}'+{\sf D},
\end{split}
\]
having used the fact that $f_n$, being 1-Lipschitz w.r.t.\ $\sfd_n'$, has  oscillation  $\leq {\sf D}'$. It follows that for $n$ sufficiently large we have
\[
\begin{split}
\Big|\int f_n-g_n\circ p_n\,\d\mu_{i,n}\Big|&\leq \int |f_n-g_n\circ p_n|\,\d\mu_{i,n}\leq \eps+\mu_{i,n}(\{|f_n-g_n\circ p_n|>\eps\})\big( \eps+{\sf D}'+{\sf D}\big),
\end{split}
\]
which, by letting first $n\uparrow\infty$ and then $\eps\downarrow0$, gives
\begin{equation}
\label{eq:fngn}
\lims_{n\to\infty}\Big|\int f_n-g_n\circ p_n\,\d\mu_{i,n}\Big|=0,\qquad i=1,2.
\end{equation}
Moreover, since $(p_n)_\#\mu_{i,n}$ is concentrated on $B_i$ for any $n\in\N$ and $i=1,2$, we have
\[
\int  g_n\circ p_n\,\d\mu_{2,n}-\int g_n\circ p_n\,\d\mu_{1,n}\leq\sup_{B_2} g_n-\inf_{B_1}g_n\leq \operatorname{diam}(B_1\cup B_2)\Lip(g_n)= {\sf D}.
\]
This estimate and \eqref{eq:fngn} imply the desired \eqref{eq:claim: dist wassdual} and thus the claim \eqref{eq:claim: dist wass} follows.

Let now  ${\ppi}'_n$ be an optimal plan for ${\W}'_1(\mu_{1, n}, \mu_{2, n})$. Note that, by Proposition \ref{prop: caratterizzazione} ${\ppi}_n'$-a.e.\ we have $\sfd_n\geq {\sf D}-\delta-\varepsilon_n$, thus also $\sfd_n'\geq {\sf D}-\delta-\varepsilon_n$ holds  ${\ppi}_n'$-a.e.\ and we have
		\begin{align}
			{\sf D}\geq \lims_{n\to\infty}{W}'_1(\mu_{1, n}, \mu_{2, n})&\geq \limi_{n\to\infty}\int_{\{\sfd_n'< {\sf D}+\delta^\frac12 \}}\sfd'_n \,\d\ppi'_n+ \lims_{n\to\infty}\int_{\{\sfd'_n\geq {\sf D}+\delta^\frac12\}} {\sfd'_n} \,\d{\ppi}'_n\\
			&\geq\limi_{n\to\infty} ({\sf D}-\delta- \varepsilon_n)\ppi'_n(\{\sfd_n'< {\sf D}+\delta^\frac12\})\\
			&\qquad\qquad+({\sf D}+\delta^\frac12)\lims_{n\to\infty}\ppi'_n(\{\sfd_n'\geq{\sf D}+\delta^\frac12\})\\
			&\geq {\sf D}-\delta +(\delta+\delta^\frac12)\lims_{n\to\infty}\ppi'_n(\{\sfd_n'\geq {\sf D}+\delta^\frac12\})
		\end{align}
and thus
\begin{equation}\label{eq:perstimaTV}
\lims_{n\to\infty}{\ppi}'_n(\{\sfd_n\geq {\sf D}+\delta^\frac12 \})\leq 3\delta^\frac12.
\end{equation}	
Let now $F_n\coloneqq \{\sfd_n<{\sf D}+\delta^\frac12\}\subset\X_n^2$ and put ${\mu}'_{i, n}\coloneqq (\text{proj}_i)_\#\big({\ppi}'_n\restr{F_n}\big)$, $i=1,2$.
Clearly ${\mu}'_{i, n}\leq \mu_{i, n}$ and thus \eqref{eq:perstimaTV} gives
		\begin{equation}\label{eq: stima TV}
			\lims_{n\to\infty}\|\mu_{i, n}-{\mu}'_{i, n}\|_{\sf TV}=\lims_{n\to\infty}{\ppi}'_n(\X_n^2\setminus F_n)=3\delta^\frac12,\qquad i=1,2.
		\end{equation}
		Set $\ppi_n\coloneqq {\ppi}'_n\restr{F_n} +(\Delta_n)_\# (\mu_{1, n}-{\mu}'_{1, n})$, where $\Delta_n: \X_n\to \X_n\times \X_n$ is the diagonal map sending $x$ to $(x,x)$ and let $\nu_n\coloneqq (\text{proj}_2)_\#\ppi_n$. 
		Let us check that $\nu_n$ is as in the statement. We have
		\[
		\nu_n=\mu_{2,n}'+(\mu_{1, n}-{\mu}'_{1, n})\leq \mu_{2,n}+\mu_{1,n},
		\]
thus \eqref{eq:boundnun1} holds (with $\eps=0$). Also, we have
\[
\|\nu_n-\mu_{2,n}\|_{\sf TV}\leq \|\nu_n-\mu'_{2,n}\|_{\sf TV}+\|\mu'_{2,n}-\mu_{2,n}\|_{\sf TV}=\|\mu'_{1,n}-\mu_{1,n}\|_{\sf TV}+\|\mu'_{2,n}-\mu_{2,n}\|_{\sf TV}, 
\]
thus from \eqref{eq: stima TV} we get the desired \eqref{eq:boundTV1}. Finally, the construction ensures that $\mu_{1,n}=(\text{proj}_1)_\#\ppi_n$, thus we have
\[
\begin{split}
\W_2^2(\mu_{1,n},\nu_n)\leq \int_{\X_n^2}\sfd_n^2\,\d\ppi_n=\int_{F_n} \sfd_n^2\,\d\ppi'_n\leq ({\sf D}+\delta^\frac12)^2\qquad\forall n\in\N,
\end{split}
\]
so that \eqref{eq:boundW21} follows.

It remains to remove the assumption that  $\mu_{i,n}$ is concentrated on $p_n^{-1}(B_i)\cap K_n$. To this aim, fix $\eps,\eps'>0$ and find open sets $U_1,U_2\subset\X$ containing $B_1,B_2$ respectively and such that $\operatorname{diam}(U_1\cup U_2)\leq {\sf D}+\eps'$, ${\rm diam} U_i<\frac{\delta}{2}$. Then Proposition \ref{prop: caratterizzazione}, the fact that $\mu_i$ is concentrated on $B_i\subset U_i$,  that $U_i$ is open and the weak convergence of $(p_n)_\#\mu_{i,n}$ to $\mu_i$ give 
\begin{equation}
\label{eq:cin}
\limi_nc_{i,n}=\limi_n(p_n)_\#\mu_{i,n}(U_i)\geq \mu_i(U_i)=1\qquad \text{where}\qquad c_{i,n}:=\mu_{i,n}(p_n^{-1}(U_i)\cap K_n)\leq 1
\end{equation}
We thus introduce the measures $\tilde\mu_{i,n}:=\tfrac1{c_{i,n}}\mu_{i,n}\restr{p_n^{-1}(U_i)\cap K_n}$ and notice that \eqref{eq:cin} easily implies that these are uniformly integrable and satisfy $(p_n)_\#\mu_{i,n}\weakto\mu_i$, $i=1,2$. Since by construction the measure $\tilde\mu_{i,n}$ is concentrated on $p_n^{-1}(U_i)\cap K_n$, we can apply what already proved and obtain the existence of measures $\tilde\nu_n\in \pr(\X_n)$ such that
		\begin{subequations}
		\begin{align}
		\label{eq:boundW21t}
		\lims_{n\to\infty} \W_2(\tilde\mu_{1, n},\tilde \nu_n)&\leq (1+\delta^\frac12)({\sf D}+\eps')+5\delta^\frac12,\\
		\label{eq:boundTV1t}
		\lims_{n\to\infty} \|\tilde\nu_n-\tilde\mu_{2, n}\|_{\sf TV}&=3\delta^\frac12 ,\\
		\label{eq:boundnun1t}
		\tilde\nu_n&\leq \tilde\mu_{1, n}+\tilde\mu_{2, n},\qquad\forall n\in\N.
		\end{align}
		\end{subequations}
Put $\nu_n:=c_{1,n}\tilde\nu_n+\mu_{1,n}\restr{\X_n\setminus (p_n^{-1}(U_1)\cap K_n)}$ and notice that \eqref{eq:boundnun1t} gives
\begin{equation}
\label{eq:bnun}
\nu_n\leq c_{1,n}\tilde\mu_{1, n}+\tilde\mu_{2, n}+\mu_{1,n}\restr{\X_n\setminus p_n^{-1}(U_1)\cap K_n}=\mu_{1,n}+\tilde\mu_{2,n}\leq \mu_{1,n}+\tfrac1{c_{2,n}}\mu_{2,n}
\end{equation}
and thus by \eqref{eq:cin} and recalling that $c_{2,n}$ depends on $U_2$ and thus on $\eps'$ we get
\begin{equation}
\label{eq:bnun2}
\nu_n\leq  \mu_{1,n}+(1+\eps)\mu_{2,n}\qquad\text{for every $n$ sufficiently big (depending on $\eps,\eps'$).}
\end{equation}
Also,  \eqref{eq:boundTV1t} and \eqref{eq:cin} imply \eqref{eq:boundTV1} and  writing $\mu_{1,n}=c_{1,n}\tilde\mu_{1,n}+\mu_{1,n}\restr{\X_n\setminus p_n^{-1}(U_1)}$ and recalling the joint convexity of $\W_2^2(\cdot,\cdot)$ and  \eqref{eq:boundW21t} we get
\begin{equation}
\label{eq:quasi}
\lims_{n\to\infty}\W_2^2(\mu_{1,n},\nu_n)\leq \lims_{n\to\infty} c_{1,n}\W_2^2(\tilde\mu_{1, n},\tilde \nu_n)\leq ( (1+\delta^\frac12)({\sf D}+\eps')+5\delta^\frac12)^2.
\end{equation}
The conclusion follows by diagonalization letting $\eps'\downarrow0$.
\end{proof}

\begin{lm} \label{lm: appr discretizzazione}
	With the same assumptions and notation as in Theorem 4.1 the following holds. Assume that:
\begin{itemize}
\item[i)] $n\mapsto\mu_n\in \pr(\X_n)$ is an equi-integrable sequence weakly converging to $\mu\in \pr(\X)$ 
\item[ii)] $\mu=\sum_{i=1}^I\alpha^i\mu^i$ for some $I\in\N$, $\alpha^i\geq 0$ with $\sum_i\alpha^i=1$ and $(\mu^i)\subset \pr(\X)$.
\end{itemize}
Then there are $\mu^i_n\in \pr(\X_n)$ and $\alpha^i_n\geq 0$, $n\in\N$, $i=1,\ldots,I$, such that
\begin{itemize}
\item[i')] for any $i$ we have $\alpha^i_n\to\alpha^i$ and the sequence $n\mapsto\mu^i_n\in \pr(\X_n)$ is  equi-integrable and weakly converging to $\mu^i\in \pr(\X)$ 
\item[ii')] for any $n\in\N$ we have $\mu_n=\sum_{i=1}^I\alpha^i_n\mu_n^i$.
\end{itemize}
\end{lm}
\begin{proof} We can assume $\alpha^i>0$ for every $i$. Let $f^i:=\frac{\d\mu^i}{\d\mu}\in \L^\infty(\mu)$ and, for fixed $i$, let $(\varphi^{i,k})\subset C^0_\b(\X)$ be a uniformly bounded sequence of non-negative functions $\mu^i$-a.e.\ converging to $f^i$. Replacing $\varphi^{i,k}$ with $\frac{\varphi^{i,k}}{\sum_j\int_{\X}\varphi^{j,k}\,\d\mu}$ we can assume that $\sum_i\int_{\X}\varphi^{i,k}\,\d\mu=1$ for every $k$ and possibly disregarding a finite number of these, we can assume that $\alpha^{i,k}:=\int\varphi^{i,k}\,\d\mu>0$ and define $\mu^{i,k}:=(\alpha^{i,k})^{-1}\varphi^{i,k}\mu$. The assumption $(i)$ tells that $\alpha^{i,k}_n:=\int\left(\varphi^{i,k}\circ p_n\right)\,\d\mu_n\to  \alpha^{i,k}>0$ so for $n\gg1$ we can define  $\mu^{i,k}_n:=(\alpha^{i,k}_n)^{-1}\left(\varphi^{i,k}\circ p_n\right)\,\mu_n$. Then the trivial identity $(p_n)_\#(\varphi^{i,k}\circ p_n\,\mu_n)=\varphi^{i,k}\,(p_n)_\#\mu_n$ shows that $\mu^{i,k}_n\weakto \mu^{i,k}$.

Since the construction ensures that $\sum_i\alpha^{i,k}\mu^{i,k}_n=\mu_n$ for every $k,n$ and that $\mu^{i,k}\weakto \mu^i$ for any $k$, with a diagunalization argument we conclude (the fact that the $\varphi^{i,k}$'s are uniformly bounded ensures that the resulting sequence of measures is equi-integrable).
\end{proof}

	\begin{proof}[Proof of Theorem \ref{teo: pullback guidato}]
	By a diagonalization argument to conclude it suffices to find a function $\omega:(0,\infty)\to(0,\infty)$ such that $\lim_{z\downarrow0}\omega(z)=0$ with the property  that for any $\eps'>0$ we can find $\nu_n\in \pr(\X_n)$ such that
			\begin{subequations}
		\begin{align}
		\label{eq:boundW2p}
		 \lims_{n\to\infty}&\W^2_2(\nu_n, \mu_{1, n})\leq \W^2_2(\mu_2, \mu_1)+\omega(\eps') ,\\
		\label{eq:boundTVp}
		\lims_{n\to\infty}& \|\nu_n- \mu_{2, n}\|_{\sf TV}\leq\omega(\eps'),\\
		\label{eq:boundnunp}
		\nu_n&\leq \mu_{1, n}+(1+\eps )\mu_{2, n},\qquad\forall n\in\N.
		\end{align}
		\end{subequations}	
Fix $\eps'>0$,and notice that by the continuity of $\sfd^2$ on $\X^2$,  any $(x,y)\in \X^2$ has rectangular a neighbourhood $W=U_1\times U_2\subset\X^2$  such that 
\begin{equation}
\label{eq:stimetta}
\big(\operatorname{diam}(W))^2\leq\inf_{W}\sfd^2+\eps',	\quad {\rm diam} (U_i)\leq \frac{(\eps')^2}{2}, i=1, 2,
\end{equation}
then by the Lindelof property of $\X^2$ we can find a countable collection $(W_i)$ of such neighbourhoods covering the whole $\X^2$.

Put $U_j:=\cup_{i<j}W_i$, let $\ppi$ be an optimal plan for $\W_2(\mu_1,\mu_2)$ and find $j\gg1$ so that putting $\ppi_j:=\ppi\restr{U_j}$ we have
\[
\begin{split}
\alpha^0:=\ppi(\X^2\setminus U_j)&\leq\eps',\\
\Big|\int \sfd^2\,\d\ppi-\int \sfd^2\,\d\ppi_j\Big|&\leq\eps'.
\end{split}
\]
Clearly, any $j$ sufficiently big does the job.  For $i=1,\ldots, j$ let $\alpha^i:=\ppi_{j}(W_i)$, $\ppi^i_j:=(\alpha^i)^{-1}\ppi_{j}\restr{W_i}$ and then 
\[
\mu_{1}^{i}:= (\proj_1)_\#\big(\ppi^i_{j}\big),\qquad\qquad \mu_{2}^{i}:= (\proj_2)_\#\big(\ppi^i_{j}\big),
\]
(if $\alpha^i=0$ these definition are irrelevant). Let also $\mu^0_1:=(\alpha^0)^{-1}(\mu_1-(\proj_1)_\#\ppi_j)$ and similarly $\mu^0_2:=(\alpha^0)^{-1}(\mu_2-(\proj_2)_\#\ppi_j)$.

Clearly we have $\sum_i\alpha^i\mu^i_1=\mu_1$ and $\sum_i\alpha^i\mu^i_2=\mu_2$, thus we can apply twice Lemma \ref{lm: appr discretizzazione} above and get sequences $n\mapsto \mu^i_{1,n},\mu^i_{2,n}\in \pr(\X_n)$ weakly converging to $\mu_1^i,\mu^i_2$ respectively and such that $\sum_i\alpha^i\mu^i_{1,n}=\mu_{1,n}$ and $\sum_i\alpha^i\mu^i_{2,n}=\mu_{2,n}$.

For each $i=1,\ldots,j$ we apply Lemma \ref{lemma: trasporto} with $U_1=U^i_1, U_2=U^i_2$ being the closure of ${\proj_1(W_i)},\proj_2(W_i)$ respectively and obtain a sequence $n\mapsto \nu^i_n\in \pr(\X_n)$ such that
\[
\begin{split}
\lims_{n\to\infty}\W_2^2(\mu^i_{1,n},\nu^i_n)&\leq \operatorname{diam}(U^i_1\cup U^i_2)^2\stackrel{\eqref{eq:stimetta}}\leq \int (1+\eps')\sfd^2\,\d\ppi^i_j+6\eps',\\
\lims_{n\to\infty}\|\nu^i_n-\mu^i_{2,n}\|_{\sf TV}&=3\eps',\\
\nu^i_n&\leq\mu^i_{1,n}+(1+\eps)\mu^i_{2,n},\qquad\forall n\in\N.
\end{split}
\]
Defining $\nu_n:=\alpha^0\mu^0_{1,n}+\sum_{i=1}^j\alpha^i\nu^i_n$, by the joint convexity of $(\mu,\nu)\mapsto \W_2^2(\mu,\nu),\|\mu-\nu\|_{\sf TV}$ and the above we deduce that
\[
\begin{split}
\lims_{n\to\infty}\W_2^2(\mu_{1,n},\nu_n)&\leq \sum_{i=1}^j\alpha^i\lims_{n\to\infty}\W_2^2(\mu^i_{1,n},\nu^i_n)\leq  \int (1+\eps')\sfd^2\,\d\ppi_j+6\eps',\\
\lims_{n\to\infty}\|\nu_n-\mu_{2,n}\|_{\sf TV}&=3\eps',\\
\nu_n&\leq\mu_{1,n}+(1+\eps)\mu_{2,n},\qquad\forall n\in\N,
\end{split}
\]
thus concluding the proof.
\end{proof}

\section{\texorpdfstring{$\Gamma$--$\lims$}{} of the Cheeger energy} \label{section: limsup}
Given $(\X, \sfd, \mm)$ a n.m.m.s., the natural object extending the Dirichlet energy is the Cheeger energy, defined by approximation with Lipschitz functions. Indeed, if $f\in \L^2(\X)$, its Cheeger energy is defined as
	\begin{equation}\label{eq: Cheeger}
		\ch(f)\coloneqq \inf\{\limi_n \frac{1}{2}\int_\X lip_a(f_n)^2\,\d\mm:\; f_n\in \Lip(\X), f_n\xrightarrow{\L^2}f \},
	\end{equation}
	where $lip_a(f)(x)\coloneqq\lims_{r>0} \Lip_{B(x, r)}$.
For $f\in W^{1, 2}(\X)\coloneqq \D(\ch)$, moreover, there exists $|\D f|\in \L^2(\X)$, called minimal weak upper gradient, such that $\ch(f)\coloneqq \frac{1}{2}\int_\X |\D f|^2\, \d\mm$ and $|\D f|=\L^2-\lim \operatorname{lip}_a(f_n)$ for some sequence $f_n$ as above.

In the setting of mGH convergence, it is well known that a $\Gamma$--$\lims$ inequality holds in the $\L^2$ strong topology for the Cheeger energy, with no additional assumptions on the spaces.
Here we prove that the same result holds when the spaces are only converging in concentration. Given a sequence of n.m.m.s. $\X_n$ such that $\X_n\concto\X$ and a sequence of projections $(p_n)_{n\geq 1}$ as in Proposition \ref{prop: caratterizzazione}, the following holds:

\begin{prop}\label{proposition: limsup}
	Let $f\in W^{1, 2}(\X)$. Then there exists a sequence $(f_n)_{n\geq 1}$ with $f_n\in \L^2(\X_n)$ and $f_n\xrightarrow{{\L^2}} f$ such that
	\begin{equation}
		\lims_n \ch(f_n)\leq \ch(f).
	\end{equation}
\end{prop}

We will provide two slightly different proofs of this result: one relies on Theorem \ref{teo: unicità limite} and the important \cite[Proposition 6.12]{Sh16}, while the other relies on a more direct construction.
\begin{proof}
	By \cite[Proposition 6.12]{Sh16} there exists a sequence of spaces $(\X'_n, \sfd'_n, \mm_n')$ converging in the mGH topology to $\X$ and a sequence of maps $q_n:\X_n\to \X'_n$ such that:
	\begin{itemize}
		\item the maps $q_n$ are 1-Lipschitz;
		\item it holds $\left(q_n\right)_\# \mm_n=\mm'_n$.
	\end{itemize}
	By mGH convergence, we might assume that $\X'_n\subset\X$ for all $n$ (up to enlarging $\X$ and without modifying $\supp \,\mm$). It is now easy to check that the sequence $(q_n)_{n\geq 1}$ is as in Proposition \ref{prop: caratterizzazione}. By Theorem \ref{teo: unicità limite} with $\Phi=id$, we get that there exists a sequence of isomorphisms $\Phi_n\in {\rm Iso}(\X_n)$ such that the projections $p'_n\coloneqq \Phi_n\circ q_n$ are equivalent to $p_n$ in the sense of Remark \ref{oss: equivalenza}. In particular, limits taken with respect to $(p_n)_{n\geq 1}$ or $(p'_n)_{n\geq 1}$ are the same. 

	Now, by \cite[Theorem 4.4]{Gi23} there exists a sequence $(f_n')_{n\geq 1}$ with $f_n'\in \L^2(\X'_n)$ such that $\lims_n \ch(f_n')\leq \ch(f)$. Let $f_n\coloneqq f'_n\circ p'_n$. On one hand, it is immediate that 
	\begin{equation}
		f'_n\xrightarrow[p'_n]{\L^2}f, 
	\end{equation}
	and so, by Remark \ref{oss: conseguenza equivalenza}, 
	\begin{equation}
		f'_n\xrightarrow[p_n]{\L^2}f. 
	\end{equation}
	On the other hand, by the 1-lipschitzianity of $p'_n$ it holds that $\ch(f_n)\leq \ch(f'_n)$, from which the thesis follows.
\end{proof}
As in the proof, in order to provide a recovery sequence for the $\Gamma-\lims$ it is enough to have a sequence of $1$--Lipschitz projections $p'_n$ which is equivalent to $p_n$. Thankfully, \ref{lemma: Lipschitz appr} does exactly that job. In order to have enough space for the approximation, we assume that $\X$ is a subset of $\ell^\infty$ (which can always be done, up to isometry). 

\begin{cor}\label{cor: proiezioni equivalenti}
	Let $\X_n\concto \X$ and let $p_n:\X_n\to \X\subset \ell^\infty$ be projection as in Proposition \ref{prop: caratterizzazione} realizing the convergence. Then there exists a sequence of 1-Lipschitz maps $p'_n:\X_n\to \ell^\infty$ realizing the convergence and equivalent to $p'_n$ in the sense of Remark \ref{oss: equivalenza}.
\end{cor}
\begin{proof}
	It follows immediately from Lemma \ref{lemma: Lipschitz appr}.
\end{proof}
\begin{proof}[Second proof of Proposition \ref{proposition: limsup}]
The proof is the same as the first one, but with $p'_n$ as in Corollary \ref{cor: proiezioni equivalenti}.	
\end{proof}

We want to remark that this result was already obtained \cite{OzYo19} under the hypothesis of a uniform $\CD$ condition, which we drop here.

\section{Stability of the  heat flow under a curvature condition}\label{section: gamma conv slope}
	\subsection{Gradient flows on nonsmooth spaces}
	Optimal transport is an invaluable tool in order to study probabilty measures on a metric space $\X$: it allows to turn $\pr(\X)$, or some distinguished subset of it, into a metric space, the Wasserstein space. Some good reference about this topic are \cite{Vi09} or \cite{AmGiSa08}.

	One of the main reasons of the widespread study of the Wasserstein spaces is their importance in the study of various gradient flows, often related to the solution of diffusion equation (as will be evident later). When the space $\X$ is smooth, the Wasserstein space $\pr_2$ can be endowed with a (at least formal) Riemannian structure (see for instance \cite{Ot01}), which allows to carry on the study of gradient flows as in the finite dimensional case. When $\X$ is nonsmooth, though, a different notion has to be employed. The formulation which allows to generalize the theory is due to De Giorgi and is based on the observation that the definition of a gradient flow is equivalent to an inequality regarding the dissipation of energy along the evolution, the so called \emph{EDE} (in)equality (a more precise discussion and references can be found in \cite{AmGiSa08} and \cite{Gi23}). Here we recall some of the most relevant definitions and properties which will be of use in the study of the heat flow, one of the main objectives of this work.
	
	The main family of functional which are of interest in order to study gradient flows, at least from our perspective, is that of geodesically $K$\emph{--convex functionals}. From here on, $\X$ is a complete and separable metric space.
	\begin{deff}[Geodesically $K$--convex functional]
		A functional $E:\X\to [0, +\infty]$ is said to be geodesically $K$--convex if for all $x, y\in \D(E)$ there exists $\gamma:[0, 1]\to \X$ a geodesic between $x$ and $y$ such that
		\begin{equation}
			E(\gamma_t)\leq (1-t)E(x)+tE(y)-\frac{1}{2}Kt(1-t)d^2(x, y)
		\end{equation}
		holds.
	\end{deff}
	In order to extend the definition of gradient flow, we need a quantity mimicking the ``speed at which $E$ decreases'', i.e. the (descending) slope.
	\begin{deff}[Slope]
		Let $E:\X\to [0, +\infty]$.The slope $|\partial^-E|(\cdot)$ of $E$ is defined as follows: 
		\begin{equation}
			|\partial^-E|(x)=\begin{cases}				
				\lims_{y\to x} \frac{(E(x)-E(y))^+}{\sfd(x, y)}\qquad &\text{if }x\in \D(E)\\
				+\infty\qquad &\text{if }x \notin \D(E).
			\end{cases}
		\end{equation}
	\end{deff}
	It is important to notice that whenever $E$ is $K$--geodesically convex, then
	\begin{equation}\label{eq: sup slope}
		|\partial^-E|(x)=\sup_y \frac{\left(E(x)-E(y)+\frac{K}{2}\sfd(x, y)^2\right)^+}{d(x, y)}.
	\end{equation}
	Finally, recall that if $\gamma:[0, 1]\to \X$ is a curve, we say that $\gamma$ is absolutely continuous whenever there exists $v\in \L^1(0, 1)$ such that 
	\begin{equation}
		|\gamma_s-\gamma_t|\leq \int_s^t v_r\,\d r
	\end{equation}
	for all $s\leq t\in [0, 1]$. In this case, for a.e. $t\in(0, 1)$ there exist $|\dot{\gamma}_t|\coloneqq\lim_{h\to 0}\frac{\d(\gamma_{t-h}, \gamma_t)}{|h|}$, which is called the \emph{metric speed} of $\gamma$.
	
	Given all these notions, it is easy to see that, given a lower semicontinuous and $K$--geodesically convex functional $E$ and any curve $\gamma$, it holds
	\begin{equation}\label{eq: EDE inversa}
		E(\gamma_0)\leq E(\gamma_t)+\frac{1}{2}\int_0^t |\dot{\gamma}_s|^2\,\d s+\frac{1}{2}\int_0^t|\partial^-E|(\gamma_s)^2\,\d s
	\end{equation}
	for all $t>0$.
	On the other hand, in any smooth setting, $\dot{\gamma_t}=-\nabla E(\gamma_t)$ implies that the equality holds in \eqref{eq: EDE inversa}. The $EDE$ gradient flow formulation reflects this observation.
	\begin{deff}[Gradient flow of the entropy functional]
		Let $E:\X\to [0, +\infty]$ a lower semicontinuous and $K$--convex functional. We say that a curve $\gamma:[0, +\infty)\to \X$ is a \emph{$EDE$ gradient flow} for $E$ starting at $x$ if it is absolutely continuous and
		\begin{equation}\label{eq: EDI}
			E(x)\geq E(\gamma_t)+\frac{1}{2}\int_0^t|\dot{\gamma}_r|^2\,\d r+\frac{1}{2}\int_0^t |\partial^-E|(\gamma_r)^2\,\d r
		\end{equation}
		for all $t>0$.
	\end{deff}
	\begin{rmk}\label{oss: EDE}
		If $\gamma_t$ is a gradient flow for $E$ starting at $x$, then
		\begin{equation}
			E(x)= E(\gamma_t)+\frac{1}{2}\int_0^t|\dot{\gamma}_r|^2\,\d r+\frac{1}{2}\int_0^t |\partial^-E|(\gamma_r)^2\,\d r
		\end{equation}
		for all $t\geq 0$.
	\end{rmk}
	
	In this section we deal with the problem of establishing $\Gamma$--convergence of the slope $\slope{\mm_n}{\cdot}$ of the Boltzmann entropy functional and Mosco convergence of the Dirichlet energy functional. Building on that, we show that the gradient flow of the entropy and the heat flow (which are strictly related) are stable with respect to convergence in concentration of the domain spaces $\X_n$.
	\subsection{CD and RCD spaces}
	Here we recall the definition of Curvature--Dimension condition (in the infinite dimensional case) and some basic functional analytic properties.
	
	Given a n.m.m.s. $(\X, \sfd, \mm)$, the \emph{Boltzmann entropy functional} $\ent: \pr(\X)\to [0, +\infty]$ is defined as
	\begin{equation}
		\ent(\mu)\coloneqq \begin{cases}
			\int_\X \rho\log(\rho)\,\d\mm\qquad&\text{if }\mu=\rho\,\mm;\\
			+\infty &\text{ if else.}
		\end{cases}
	\end{equation}
	\begin{deff}
		We say that a n.m.m.s. $(\X, \sfd, \mm)$ satisfies the $\CD(K, \infty)$ condition for $K\in\mathbb R$ if $\ent$ is $K$--geodesically convex.
	\end{deff}
	The $K$-- geodesical convexity of the entropy functional allows to define (mimicking what happens in the smooth case) a notion of heat flow as the gradient flow of $\rm{Ent}$.
	\begin{thm}\label{teo: monot slope}
		Let $(\X, \sfd, \mm)$ be a $\CD(K, \infty)$ n.m.m.s. and $\mu\in \D(\ent)$. Then there exist a unique $\W_2$--gradient flow of the entropy $\mu_t\eqqcolon\hgf{\mu}$ starting from $\mu$. Moreover:
		\begin{itemize}
			\item $|\dot{\mu_t}|=\slope{m}{\mu_t}$ for a.e. $t>0$;
			\item $t\mapsto e^{Kt}\slope{m}{\mu_t}$ is l.s.c. and non increasing.
		\end{itemize}
	\end{thm}
	
	%  Given $f\in \L^2(\X)$, its Cheeger energy is defined as
	% \begin{equation}
	% 	\ch(f)\coloneqq \inf\{\limi_n \frac{1}{2}\int_\X lip_a(f_n)^2\,\d\mm:\; f_n\in \Lip(\X), f_n\xrightarrow{\L^2}f \},
	% \end{equation}
	% where $lip_a(f)(x)\coloneqq\lims_{r>0} \Lip_{B(x, r)}$.
	% For $f\in W^{1, 2}(\X)\coloneqq \D(\ch)$, moreover, there exist $|\D f|\in \L^2(\X)$, called minimal weak upper gradient, such that $\ch(f)\coloneqq \frac{1}{2}\int_\X |\D f|^2\, \d\mm$ and $|\D f|=\L^2-\lim \operatorname{lip}_a(f_n)$ for some sequence $f_n$ as above.

	On the other hand, a heat flow can be defined also via $\L^2(\X)$ as a gradient flow of the Cheeger energy.
	Given the definition in \eqref{eq: Cheeger}, it is easy to check that $\ch$ is convex and lower semicontinuous, so another notion of heat flow $h_t$ can be defined as gradient flow of $\ch$ via the classical theory of gradient flows in Hilbert spaces.
	In case $W^{1, 2}$ is Hilbert (which amounts to $\ch$ being a quadratic form), $\X$ is said to be infinitesimally hilbertian: $h_t$ can be shown to be linear, the subdifferential of $\ch$ at every $f\in \L^2$ contains at most one element $\Delta f$, which is called Laplacian of $f$ as in the smooth case, and $\Delta$ is a linear unbounded operator.
	
	This two seemingly competing definitions of the heat flow are indeed the same.	
	\begin{thm}[Slope of $\ent$ and Cheeger energy]\label{teo: Slope=Cheeger}
		Let $\X$ be a $\CD(K, \infty)$ n.m.m.s. and let $f\in \L^2(\X)$ with $\|f\|_{\L^2}=1$.
		Then 
		\begin{equation}\label{eq:slopecheeger}
			\slope{\mm}{f^2\mm}=8 \ch(f).
		\end{equation}
		Moreover $\mathcal H_t(f^2 \mm)=(h_t f)\mm$ for all $t>0$.
	\end{thm}
	% \add{osservazione su $\pr_2$ vs $\pr$?.}

	Stability properties of these quantities are well known in the setting of measured Gromov--Hausdorff convergence. Since the very beginning of the theory (\cite{LoVi09}, \cite{St06-1}, \cite{St06-2}) convergence of the entropies and stability of the $\CD(K, N)$ conditions have been established. Somewhat more recently, convergence of the slope of the entropy and of the heat flow have been proved in \cite{GiMoSa15}. In the setting of convergence in concentration, these properties were only partially investigated. We recall here some of the known results.	
	\begin{thm}[\cite{OzYo19}, cfr. \cite{FuSh13}]\label{teo: Gamma ent}
		Let $(\X_n, \sfd_n, \mm_n), (\X, d, m)$ be n.m.m.s.. If $\X_n\to \X$ in concentration, then $\entm{\mm_n}\xrightarrow{\Gamma}\ent$, with respect to weak convergence of measures.
	\end{thm}
	For later use, we state here this really simple lemma.
	\begin{lm}\label{lemma: approssimazione}
		Let $\X_n, \X$ be n.m.m.s. and let $p_n, \tilde{\X}_n, \varepsilon_n$ be as in Proposition \ref{prop: caratterizzazione}.
		If $\mu\in \pr(\X)$ is such that $\mu\leq C \mm$, there exist a recovery sequence $\mu_n$ such that $\mu_n\leq 2C \mm_n$: it is enough to approximate $\frac{\d\mu}{\d\mm}$ with continuous functions $\varphi_j\leq 2C$ such that $\int \varphi_j\, \d\mm=1$ and $\int \varphi_j\log(\varphi_j)\,\d\mm\to \ent(\mu)$; finally, a diagonal argument allows to chose a suitable $\mu_n=\varphi_{j_n}\circ p_n \mm_n$.
	\end{lm}
	The $\lims$ inequality can be generalized to this result, which was the inspiration for Lemma \ref{lemma: trasporto}, of fundamental importance in this paper.
	\begin{lm}\cite{OzYo19}\label{lm: approssimazione}
		Let $(\X_n, \sfd_n, \mm_n), (\X, \sfd, \mm)$ be n.m.m.s. with $\X_n\to \X$ in concentration. For every $\mu, \nu\in \pr(\X)$ and $p\in [1, +\infty)$ there exist $\mu_n, \nu_n \in \pr(\X_n)$ such that $\entm{\mm_n}(\mu_n)\to \ent(\mu)$, $\entm{\mm_n}(\nu_n)\to \ent(\nu)$ and $\W_p(\mu_n, \nu_n)\to \W_p(\mu, \nu)$.
	\end{lm}
%	\begin{teo}[\cite{FuSh13}]
%		Let $(\X_n, d_n, m_n), (\X, d, m)$ be n.m.m.s. with $\X_n\to \X$ in concentration. If $\X_n$ is $\CD(K, \infty)$ for all $n\geq 1$ and some $K\in\mathbb R$, then $\X$ is $\CD(K, \infty)$.
%	\end{teo}
%	\begin{teo}[\cite{OzYo19}]
%		Let $(\X_n, d_n, m_n), (\X, d, m)$ be n.m.m.s. with $\X_n\to \X$ in concentration and with $\X_n$ satisfying $\CD(K, \infty)$ for all $n\geq 1$. Then $\ch_n \xrightarrow{\Gamma} \ch$ with respect to strong $\L^2$ convergence. In particular, if $\X_n$ is $R\CD(K, \infty)$ for all $n\geq 1$ then $\X$ is as well.
%	\end{teo}
	Finally, we conclude this part with some simple technical results that we will use in the following.	
	\begin{lm} \label{lemma: tight}
		Given $\varepsilon>0$ and $C>0$ there exists $\delta=\delta(\varepsilon, C)$ such that, whenever $(\Y, \mm)$ is a probability space, $\mu\in \pr(\Y)$ and $\ent(\mu)\leq C$ then $\mm(E)<\delta$ implies $\mu(E)<\varepsilon$ for every $E\in\mathcal{B}(\Y)$.
	\end{lm}
	\begin{proof}
		Let $\mu=\rho \mm$. For $M>1$ holds 
		\begin{align}
			\mu(E)&=\int_E \rho\, \d\mm=\int_{E\cap\{\rho\geq M\}}\rho\, \d\mm+\int_{E\cap\{\rho<M\}}\rho\, \d\mm\\
			&\leq \int_{E\cap\{\rho\geq M\}}\frac{\rho\log(\rho)+1}{\log(M)}\, \d\mm+M\mm(E)\\
			&\leq \frac{C+1}{\log(M)}+M\delta.
		\end{align}
		Choosing $M=\delta^{-\frac{1}{2}}$, it follows $\mu(E)\leq  2\frac{C+1}{|\log(\delta)|}+\delta^{\frac{1}{2}}$. Finally, by choosing $\delta$ small enough, the thesis holds.
	\end{proof}
	\begin{cor}
		Let $(\X_n, \sfd_n, \mm_n)\to (\X, \sfd, \mm)$ in concentration and let $\mu_n\in \pr (\X_n)$ be such that $\sup_n \entm{\mm_n}(\mu_n)<+\infty$. Then $(\mu_n)$ is equi--integrable.
	\end{cor}
	\begin{lm}[\cite{OzYo19}]\label{lemma: semicontinuità uniforme}
		Let $\X$ be a topological space and $\mm, \mu, \nu\in \pr (\X)$. Then for all $\varepsilon, M>0$ there exists $\delta=\delta (\varepsilon, M)$ (not depending on $\X$) such that whenever $\nu\leq M \mm$ and $\|\nu-\mu\|_{TV}\leq \delta$ it holds
		\begin{equation}
			\ent(\nu)\leq \ent(\mu)+\varepsilon.
		\end{equation}
	\end{lm}
	\subsection{Stability properties}\label{subsection: stab} 
	Throughout all this section we shall always assume to have a converging sequence $(\X_n,\sfd_n,\mm_n)\xrightarrow{conc}(\X,\sfd,\mm)$, all the spaces being $\CD(K,\infty)$ for some $K\in\R$. We shall also fix projection maps $p_n:\X_n\to\X$ as in Proposition \ref{prop: caratterizzazione}.
\begin{proof}[Proof of Theorem \ref{teo: liminf slope}]
Let us start with the additional assumption that for some $M>0$ we have $\mu_n\leq M\mm_n$ for every $n\in\N$. Recall that
		\begin{equation}
		\label{eq:reprslope}
 |\partial^-\ent (\mu)|=\sup_{\mu_2\neq \mu}\left(\frac{\ent (\mu)-\ent (\mu_2)}{\W_2(\mu, \mu_2)}+\frac{K}{2}\W_2(\mu, \mu_2)\right)^+
		\end{equation}
and notice that by an approximation argument it is enough to consider probability measures $\mu_2$ with bounded density. Thus fix such   $\mu_2\in \pr (\X)$ and let $\mu_{2, n}\in \pr (\X_n)$ be such that $\mu_{2, n}\rightharpoonup\mu_2$, $\entm{\mm_n}(\mu_{2, n})\to \ent (\mu_2)$ and $\big|\frac{\d\mu_{2, n}}{\d \mm_n}\big|\leq\big\|\frac{\d\mu_2}{\d\mm}\big\|_{\infty}$ (recall Lemma \ref{lemma: approssimazione}).

		Let $(\nu_n)$ be the sequence given by Theorem \ref{teo: pullback guidato} (with $\eps=1$). Since both $\mu_n$ and $\mu_{2, n}$ have equibounded densities, by \eqref{eq:boundnun} $\nu_n$  does too. Then Lemma \ref{lemma: semicontinuità uniforme} and \eqref{eq:boundTV} give $\ent (\mu_2)\geq \lims_n \entm{\mm_n}(\nu_n)$ and therefore using also \eqref{eq:boundW2} we get
\begin{align}
\frac{\ent (\mu)-\ent(\mu_2)}{\W_2(\mu, \mu_2)}+\frac{K}{2}\W_2(\mu, \mu_2)&\leq \limi_{n\to\infty}\frac{\entm{\mm_n}(\mu_n)-\ent{\mm_n}(\nu_n)}{\W_2(\mu_n, \nu_n)}+\frac{K}{2}\W_2(\nu_n, \mu_n)\\
\text{(by \eqref{eq:reprslope})}\qquad\qquad&\leq \limi_{n\to\infty}  |\partial^-\entm{\mm_n} (\mu_n)|.
\end{align}
By the generality of $\mu_2$ the conclusion follows.

It remains to prove that we can reduce to the case of the $\mu_n$'s having uniformly bounded density.  For $M>0$ define   $\mu^M_n=c_{n, M}\rho_n^M\,\mm_n$, where $\mu_n=\rho_n \mm_n$ and $\rho^M_n\coloneqq \rho_n\wedge M$ and $c_{n,M}$ is the normalization constant. The equi--integrability of $(\mu_n)$ and Lemma \ref{lemma: troncamenti} gives $\lim_{M\to \infty} \sup_n c_{n, M}=1$ and then Lemma \ref{lemma: compattezza equi-int}  yield that, up to a subsequence in $n$, we have
\[
\begin{split}
\mu_n^M &\xrightharpoonup[n\to\infty]{} \mu^M\qquad \text{for some }\quad \mu^M\in\pr(\X)\quad \text{and every }M\in\N,\\
\mu^M&\xrightharpoonup[M\to \infty]{} \mu.
\end{split}
\]
In particular we have  $\mu^M_n\leq 2M\mm_n$ for every $n\in\N$ provided $M$ large enough and thus the above argument yields $|\partial^-\ent (\mu^M)|\leq\limi_n|\partial^-\entm{\mm_n} (\mu^M_n)|$. Since $|\partial^-\ent (\mu)|\leq \limi_M|\partial^-\ent (\mu^M)|$ holds by lower semicontinuity of the slope on $\X$, the conclusion follows from the key identity \eqref{eq:slopecheeger}  and the fact that Cheeger energy decreases when truncating.
\end{proof}

	Now, building on the $\Gamma-\limi$ of the slope of the entropy, we can recover most of the stability results proven in \cite{Gi10} and \cite{GiMoSa15} for mGH convergence, see also the survey \cite{Gi23}.
	
	We start with the stability of the gradient flow of the entropy.
	\begin{thm}[Convergence of the gradient flow of $\ent$] \label{teo: conv heat}
		Let $\mu_n\in\pr(\X_n)$ be weakly converging to $\mu$ and  such that $\entm{\mm_n}(\mu_n)\to \ent(\mu)<+\infty$.
		
		Let $t\mapsto\mu_{n,t}$ be the (only, by \cite{Gi10}, see also the presentation in \cite{AmGiSa14}) $\W_2$-gradient flow of $\entm{\mm_n}$ starting from $\mu_n$, and analogously $t\mapsto\mu_t$.
		Then 
		\begin{itemize}
			\item  $\mu_{n,t} \rightharpoonup\mu_t$ for all $t>0$;
			\item $\entm{\mm_n}(\mu_{n,t})\to\ent(\mu_t)$ for all $t>0$;
			\item $|\partial^-\entm{\mm_n} (\mu_{n,t})|\to |\partial^-\ent (\mu_t)|$ for all $t\in(0, +\infty)\setminus \mathcal S$, where $\mathcal S$ is the discontinuity set of $t\mapsto |\partial^-\ent (\mu_t)|$.
		\end{itemize}
	\end{thm}
	\begin{proof} We closely follow the arguments in \cite{Gi10}. By the energy dissipation inequality \eqref{eq: EDI}, the curves are equibounded in entropy and equi--$\W_2$--H\"older. Thus by \eqref{eq: semicont W2} and an Ascoli--Arzelà-type argument, it is easy to see that any subsequence of the curves has a weakly convergent sequence. Up to the extraction of a non relabelled subsequence, let us suppose that $\mu_{n, t}\rightharpoonup \nu_t$ for all $t\geq 0$ and some $(\nu_t)\subset\pr(\X)$. Let us start by proving that $(\nu_t)$ is a gradient flow of $\ent$, which,  by the uniqueness of the gradient flow, yields $\nu_t=\mu_t$ for every $t\geq 0$. The representation formula 
	\[
	\int_0^t |\dot{\mu}_r|^2\,\d r=\sup\sum_i\frac{\W_2^2(\mu_{t_i},\mu_{t_{i+1}})}{2(t_{i+1}-t_i)},
	\]
the sup being taken among all $I\in \N$ and partitions $0=t_0<\cdots<t_I=t$ of $[0,t]$, and again \eqref{eq: semicont W2} yield that	
		\begin{equation}\label{eq: lsc energy}
			\int_0^t |\dot{\nu}_r|^2\,\d r\leq \limi_{n\to\infty} \int_0^t |\dot{\mu}_{n, r}|^2\,\d r,\qquad\forall t>0.
		\end{equation}
We now see that the left hand side of \eqref{eq: EDI} converges by assumption, while the right hand side may only decrease in the limit by \eqref{eq: lsc energy} and Theorems \ref{teo: Gamma ent} and \ref{teo: liminf slope}. This shows that $(\nu_t)$ is a gradient flow of $\ent$, as desired.
		
		Now, by Remark \ref{oss: EDE} and the semicontinuity results, the three terms in the right hand side of \eqref{eq: EDI} must all converge: in particular $\entm{\mm_n}(\mu_{n, t})\to\ent(\mu_t)$ for all $t>0$ and 
		\begin{equation}
				\int_0^t|\partial^-\entm{\mm_n} (\mu_{n,r})|^2\,\d r\to \int_0^t |\partial^-\ent (\mu_r)|^2\,\d r.
		\end{equation}
		By Theorem \ref{teo: liminf slope} and the monotonicity statement in Theorem \ref{teo: monot slope}, the convergence is pointwise in all continuity points of the limit.
	\end{proof}
	With the help of the convergence of the gradient flow of the entropy, we recover the full $\Gamma$--convergence of the slope.
	\begin{thm}[$\Gamma$-convergence of the slope]\label{teo: Gamma slope} For any $\mu\in\pr(\X)$, $\mu\ll\mm$, we have
	\[
\begin{split}
 |\partial^-\ent (\mu)|&\leq  |\partial^-\entm{\mm_n} (\mu_{n})|\qquad \text{for every $n\mapsto\mu_n\in\pr(\X_n)$ equi-integrable with $\mu_n\weakto\mu$,}\\
 |\partial^-\ent (\mu)|&\geq  |\partial^-\entm{\mm_n} (\mu_{n})|\qquad \text{for some $n\mapsto\mu_n\in\pr(\X_n)$ equi-integrable with $\mu_n\weakto\mu$.}
\end{split}
\]
	\end{thm}
	\begin{proof}
		The first part is precisely the content of Theorem \ref{teo: liminf slope}. The proof of the approximation result is the same as \cite[Theorem 5.14]{GiMoSa15}, that we repeat for completeness. We can assume that $|\partial^-\ent (\mu)|<\infty$ (and thus that $\ent(\mu)<\infty$) or otherwise there is nothing to prove. Use Theorem \ref{teo: Gamma ent} to find  $\mu_n\in\mathcal \pr(\X_n)$ weakly converging to $\mu$ and such that $\entm{\mm_n}(\mu_n)\to \ent(\mu)$. Use  Theorem \ref{teo: conv heat} to find  sequence $t_k\downarrow 0$ such that $ |\partial^-\entm{\mm_n} (\mu_{n,t_k})|\to |\partial^-\ent (\mu_{t_k})|$ as $n\to\infty$ for every $k\in\N$, where the curves $(\mu_{n,t}),(\mu_t)$ are the gradient flows of the entropies starting from $\mu_n,\mu$ respectively. Then   Theorem \ref{teo: monot slope} gives
\[
\begin{split}
\lim_{k\to\infty} \lim_{n\to\infty}\entm{\mm_n}(\mu_{n,t_k})&=\ent(\mu),\\
\lims_{k\to\infty} \lim_{n\to\infty}|\partial^-\entm{\mm_n}| (\mu_{n,t_k})&\leq  |\partial^-\ent |(\mu)
\end{split}
\]
and the desired recovery sequence can be built by a diagonal argument.
	\end{proof}

\begin{rmk}
	Notice also that coupling Proposition \ref{proposition: limsup} with Theorem \ref{teo: Slope=Cheeger} we could have immediately obtained the $\Gamma$--$\lims$ inequality for the slope of the entropy, even without any $\CD$ condition on the spaces.
\end{rmk}
	
	Our next goal is to prove of Mosco convergence, i.e. the $\Gamma$--convergence with respect to both the weak and the strong $\L^2$ topologies, of the Cheeger energy. From \cite{OzYo19} we already know that the $\Gamma$--convergence holds with respect to the strong topology. Their proof, though, seems to quite heavily rely on the strong convergence in order to control the oscillations ``on the fibers'' of the converging functions. In order to deal with those oscillations, we introduce a Rellich--type theorem - whose statement and proof closely follows that of \cite{GiMoSa15} - which basically allows to reduce to the strong convergence case. While doing so, we recover Ozawa--Yokota's result with slightly different methods.
	\begin{prop}[Rellich-type compactness]\label{prop: rellich}
		Let $f_n\in \L^2(\X_n,\mm_n)$ be a 2--equi--integrable sequence such that $f_n\xrightharpoonup{\L^2}f$, $\sup_n \ch_n(f_n)<+\infty$ and $\bigcup_n p_n(\supp f_n)$ is bounded. 
		
		Then $f_n\xrightarrow{\L^2} f$ and $\ch(f)\leq \limi_n \ch_n(f_n)$.
	\end{prop}
	\begin{proof} For $M>0$ let $f_n^M\coloneqq -M\vee f_n\wedge M$ and $f^M$ be any $\L^2$--weak limit point of $n\mapsto f_n^M$. Then $f^M\to f$ in $\L^2$ by 2--equiintegrability and thus taking into account the fact that the Cheeger energies decrease under truncation and the $\L^2$-lower semicontinuity of $\ch$ we may then assume that the $f_n$'s are equibounded in $\L^\infty$. Then, up to adding a constant, we shall also assume that they are nonnegative.
		
		If $\|f_n\|_{\L^2(\mm_n)}\to 0$, there is nothing more to prove, so let us assume that, up to subsequence, the norms converge to a positive constant. Up to multiplying by a converging sequence of constants and possibly passing to a subsequence, we may now assume that $\mu_n\coloneqq f_n^2 \mm_n$ is a sequence of probability measures such that:
		\begin{itemize}
			\item $\mu_n\leq C \mm_n$ for some fixed $C>0$ and all $n\geq 1$;
			\item $ \lim_{n\to\infty}|\partial^-\entm{\mm_n}|({\mu_n})\leq C$ for all $n\geq 1$ (by identity \eqref{eq:slopecheeger});
			\item $\left(p_n\right)_\#\mu_n \to \mu$ in $\W_2$ for some probability measure $\mu$ satisfying $\mu\leq C\mm$ (by weak convergence and uniform bounds on the supports).
		\end{itemize}
		Use Lemma \ref{lemma: approssimazione} to find   $\nu_n\in \pr(\X_n)$ such that $\nu_n\rightharpoonup\mu$, $\nu_n\leq 2C\mm_n$ and $\entm{\mm_n}(\nu_n)\to \ent(\mu)$. Then use Theorem \ref{teo: pullback guidato} to find perturbations $\nu'_n\in\pr(\X_n)$ such that $\W_2(\mu_n, \nu_n')\to 0$
 and, using also Lemma \ref{lemma: semicontinuità uniforme}, $\lims_n \entm{\mm_n}(\nu'_n)\leq \ent(\mu)$. Passing to the limit in
 \[
 \entm{\mm_n}(\mu_n)\leq \entm{\mm_n}(\nu_n')+ \W_2(\mu_n,\nu'_n)|\partial^-\entm{\mm_n}|({\mu_n})-\frac K2\W_2^2(\mu_n,\nu'_n)
 \]
 (that follows from the $K$--convexity of the entropy) we get  $\lims_n \entm{\mm_n}(\mu_n)\leq \lims_n \entm{\mm_n}(\nu'_n)\leq \ent(\mu)$. Since the opposite inequality $\lims_n \entm{\mm_n}(\mu_n)\geq \ent(\mu)$ always holds by Theorem \ref{teo: Gamma ent}, we can apply Lemma \ref{lemma: Young} to deduce  that $f_n^2\to g:=\frac{\d\mu}{\d\mm}$ in $\L^0$ and thus in $\L^p$ for all $p$ in $[1, +\infty)$. In particular $f_n\to g^{\frac{1}{2}}$ in $\L^0$ and so $g=f$ and $f_n\to f$ in $\L^2$.
 
		Finally, notice that $f_n\to f$ in $\L^2$ implies that $f_n^2\mm_n$ is equiintegrable and converges weakly to $f^2\mm$. The semicontinuity of the Cheeger energy follows then from Theorem \ref{teo: Slope=Cheeger} and Theorem \ref{teo: liminf slope} after renormalizing.
	\end{proof}
	\begin{cor}[Mosco convergence of Cheeger energies]\label{cor: Gamma Cheeger}
		The following two hold:
		\begin{itemize}
			\item Let $f_n\in \L^2(\mm_n)$ satisfy $f_n\rightharpoonup f$ in $\L^2$. Then $\ch(f)\leq \limi_n \ch(f_n)$.
			\item Let $f\in \L^2(\mm)$. Then there exists $f_n\in \L^2(\mm_n)$ such that $f_n\to f$ in $\L^2$ and $\ch(f_n)\to \ch(f)$.
		\end{itemize}
	\end{cor}
	\begin{proof}
		Let us start by proving the liminf inequality and thus assuming $\sup_n \ch(f_n)<\infty$.
		Reasoning as in the proof of Proposition \ref{prop: rellich}, taking into account the fact that Cheeger energy decreases when truncating and that it is lower semicontinuous wrt. weak $\L^2$ convergence, we may assume that the $f_n$ are equibounded in $\L^\infty$.
		Fix $M, R>0$ and $x\in \X$ and define $\chi_{R, M}(\cdot)\coloneqq \left(1-\frac{\sfd_n(\cdot, p_n^{-1}(B(x, R)))}{M}\right)^+$. Finally, set $f_{n, R, M}\coloneqq f_n\chi_{R, M}$. Notice that $\ch(f_{n, R, M})\leq \ch(f_n)+\frac{1}{M}\sup_n \|f_n\|^2_{\L^2(\mm_n)}$ and that $\bigcup_n p_n(\supp f_{n, R, M})$ is bounded, thanks to item (v) of Proposition \ref{prop: caratterizzazione}: by Proposition \ref{prop: rellich} $(f_{n, R, M})_{n\geq 1}$ is pre--compact in $\L^2$.
		Up to the extraction of a subsequence, we may then assume that $f_{n, R, M}\to f_{R, M}$ in $\L^2$. In particular, for all $R>0$ (except a countable quantity, at most) $f_{R, M}=f$ $\mm$--a.e. in $B(x, R)$ by construction. Finally, 
		\begin{align}
			\ch(f)&=\sup_R \int_{B(x, R)} |\d f|^2\,\d\mm\leq \sup_R \ch(f_{R, M})\\
			&\leq \sup_R \limi_n \ch(f_{n, R, M}))\leq \limi_n \ch(f_n)+\frac{1}{M}\sup_n \|f_n\|^2_{\L^2(\mm_n)}.
		\end{align}
		By the arbitrariness in the choice of $M>0$, the liminf inequality follows.
		
		In order to find a recovery sequence, by linearity and approximation, we may assume that $f\geq 0$, is bounded and $\|f\|_{\L^2}^2=1$. It is now easy to see that the square roots of the densities of the measures in the recovery sequence guaranteed by Theorem \ref{teo: Gamma slope} give the desired $\L^2$ converging sequence.
	\end{proof}
	\begin{rmk}\label{oss: Rellich}
		By reasoning as in the proof of Corollary \ref{cor: Gamma Cheeger}, it is easy to see that the condition on the projections of the supports in Proposition \ref{prop: rellich} can be weakened to 
		\begin{equation}
			\lim_{R\to \infty} \lims_n \int_{ \X_n\setminus p_n^{-1}(B(x, R))} |f_n|^2\,\d\mm_n=0
		\end{equation}
		for some $x\in \X$.
	\end{rmk}
	By exploiting the Mosco convergence result, reasoning exactly as in \cite{GiMoSa15}, we recover the convergence of the heat flow.
	\begin{deff}
		Let $f\in \L^2(\mm_n)$. Given $\tau>0$, the \emph{resolvent operator} $J_{n, \tau}$ is defined as follows:
		\begin{equation}
			J_{n, \tau}(f)\coloneqq \mathop{\operatorname{argmin}}_{g\in \L^2(\mm_n)} \frac{1}{2\tau}\|f-g\|^2_{\L^2(\mm_n)}+\ch(g).
		\end{equation}
	\end{deff}
	\begin{cor}
		Let $f_n\in \L^2(\mm_n)$ be such that $f_n\to f$ in $\L^2$. Then 
		\begin{equation}
			J_{n, \tau}(f_n)\xrightarrow{\L^2} J_\tau(f), \qquad \ch(J_{n, \tau}(f_n))\to \ch(J_\tau(f)).
		\end{equation}
	\end{cor}
	\begin{thm}\label{teo: stabilità calore}
		Let $f_n\in \L^2(\mm_n)$ be such that $f_n\to f$ in $\L^2$. Then $h_{n, t}f_n\to h_t f$ in $\L^2$.
	\end{thm}
	
	\section{Some applications} \label{section: app}
	In this last section we outline some possible applications of the stability results proven in this paper. In the following, we will always assume that for some $K\in\mathbb R$ the n.m.m.s. spaces $\X_n$, $n\geq 1$, satisfy $\RCD(K, \infty)$ and that $\X_n\to \X$ in concentration.
	
	We start by noticing that, by repeating verbatim the proof in \cite{GiMoSa15}, we recover the stability of the eigenvalues of the Laplacian. In the case of convergence in concentration, though, since Proposition \ref{prop: rellich} is a little bit weaker than the $\L^2$--compactness result in \cite{GiMoSa15}, we need a uniform log--Sobolev inequality, which is a stronger requirement than a uniform wLSTI.
	\begin{thm}
		Let $\X_n, \X$ be $\RCD(K, \infty)$ and such that $\X_n\to \X$ in concentration. Assume moreover that $\X_n$ satisfy a uniform log--Sobolev inequality, i.e. there exist $C>0$ such that
		\begin{equation}
			\int_{\X_n} f^2\log(f^2)\,\d\mm_n\leq C\int_{\X_n} |\d f_n|^2\,\d\mm_n
		\end{equation}
		for all $f_n\in \L^2(\X_n)$ such that $\|f_n\|_{\L^2(\X_n)}=1$. Then
		\begin{equation}
			\lambda_{n, k}\to \lambda_k
		\end{equation}
		for $n\to \infty$, where $\lambda_{n, k}$ is the $k$--th eigenvalue of the Laplacian on $\X_n$ and $\lambda_k$ the same on $\X$ ($\lambda_k=\infty$ if $W^{1, 2}(\X)$ has dimension less than $k$).
	\end{thm}

	Another possible application of the Mosco convergence of the Cheeger energy is in defining a meaningful notion of convergence in $W^{1, 2}$ for functions on varying spaces. Building on $W^{1, 2}$ convergence, then, one can define $\L^2$ convergence of vector and tensor fields. These ideas are already known in the setting of mGH convergence, see \cite{Gi23} for a comprehensive account.
	
	In \cite[Lemma 6.1]{Gi23}, given $f_n$ converging to $f$ in energy, the $\L^2$ convergence of the pointwise norms of the differentials $|\d f_n|$ to $|\d f|$ follows essentially from two ideas. On one hand, $\frac{\d f}{|\d f|}$ is approximated with more regular covector fields. On the other, once all the spaces have been isometrically embedded in the same $(\Y, \sfd)$, it is easy to see that any Lipschitz function $g$ on $\X$ can be extended to $\Y$ and the asymptotic Lipschitz constant of the extension is upper semicontinuous: this allows to approximate $g$ with functions $g_n$ defined on $\X_n$ with asymptotic Lipschitz constant which is not ``larger''. 
	
	In our case, since the embedding is not present, we rely on a slightly finer approximation of $\frac{\d f}{|\d f|}$ and mimic the approximating sequence $g_n$ by exploiting on the definition of convergence in concentration.
		\begin{lm}\label{lemma: mima scs}
		The set $\mathcal{R}\coloneqq\{\sum_n \mathbbm{1}_{A_n}\d g_n:\; (A_n)_n \text{are disjoint, }(g_n)_n\subset \Lip_1(\X) \}$ is strongly dense in $\{\omega\in \L^2(T\X):\, |\omega|\leq 1\}$.
	\end{lm}
	\begin{proof}
		Let us start by noticing that, by linearity of the differential, the thesis is immediate if the $g_n$'s are requested to be only Lipschitz. Now, again by linearity and continuity of the pointwise norm, we may assume that $|\omega|\equiv 1$. Let now $\omega_n\coloneqq \sum_n \mathbbm{1}_{A_n}dg_n$ be as above with Lipschitz $g_n$'s. It is clear that $\sum_n \mathbbm{1}_{A_n}|\D g_n|\to 1$ in $\L^2$: by a simple approximation argument, we may also assume that $\sum_n \mathbbm{1}_{A_n} lip_a(g_n)\to 1$ in $\L^2$. Up to another rescaling and a truncation of the sets $A_n$, we can also choose $(g_n)_n\geq 1$ to be such that $lip_a(g_n)\leq 1$: by the fact that $\X$ is a length space we conclude.
	\end{proof}
	\begin{rmk}
		One can take $g_n=\sfd(x_n, \cdot)$ for some $x_n$.
	\end{rmk}
	\begin{lm}\label{lemma: limsup lipschitz}
		Let $\X_n$ be $\RCD(K, \infty)$ and $g\in \Lip_1(\X)\cap \L^{\infty}(\X)$. Then there exist $g_n\in \Lip_1(\X_n)$ such that $g_n\to g$ in $\L^2$ and $\ch(g_n)\to \ch(g)$.
	\end{lm}
	\begin{proof}
		Let $\tilde{g}_n\in \L^2(\X_n)$ be the sequence given by Proposition \ref{prop: approssimazione e comp L0}. Without loss of generality, we may assume that $\|\tilde{g}_n\|_\infty\leq \|g\|_\infty$, so by Proposition \ref{prop: convergenze} it follows $\tilde{g}_n\to g$ in $\L^2$. It is now easy, thanks to Theorem \ref{teo: stabilità calore}, to find a sequence of times $t_n\to 0$ such that $f_n\coloneqq h_{n, t_n} \tilde{g}_n$ converges in energy to $g$. The condition on the Lipschitz constant is guaranteed by Bakry--\'Emery estimates and a suitable rescaling.		
	\end{proof}
	With these two Lemmata we are finally able to prove convergence of the minimal weak upper gradients.
	\begin{lm}\label{lemma: convergenza moduli diff}
		Let $f_n\in \L^2(\X_n)$ be such that $\sup_n \|f_n\|_{\L^\infty}+\infty<$, $f_n\to f$ in $\L^2$ and $\ch(f_n)\to \ch(f)$. Then $|\d f_n|\to|\d f|$ in $\L^2$.
	\end{lm}
	\begin{proof}
		Let us start by proving that if $g, \psi\in W^{1, 2}(\X)$ and $g_n, \psi_n\in W^{1, 2}(\X_n), n\geq 1$ are such that $g_n, \psi_n\to g, \psi$ in $\L^2$, $\sup_n \|g_n\|_{\infty}, \sup_n\|\psi_n\|_\infty<\infty$ and $\ch(g_n), \ch(\psi_n)\to \ch(g), \ch(\psi)$, then 
		\begin{equation}\label{eq: claim limite forte}
			\int \psi_n\langle \d f_n, \d g_n\rangle\,\d\mm_n\to \int \psi\langle \d f, \d g\rangle\,\d\mm.
		\end{equation}
		By the $\Gamma-\limi$ inequality in Corollary \ref{cor: Gamma Cheeger}, it holds that
		\begin{equation}
			\ch(f+\varepsilon g)\leq \limi_n \ch(f_n+\varepsilon g_n).
		\end{equation}
		By expanding the squares end exploiting the fact that $g_n$ has bounded Cheeger energy, it implies that
		\begin{equation}
			\int \langle \d f_n, \d g_n\rangle\,\d\mm_n\to \int \langle \d f, \d g\rangle\,\d\mm.
		\end{equation}
		Exploiting this and the representation 
		\begin{equation}
			\int \psi_n|\d f_n|^2\,\d\mm_n=\int \left(\langle \d f_n, \d(\psi_nf_n)\rangle-\langle \d\psi_n, \d\frac{|f_n|^2}{2}\rangle\right)\d\mm_n,
		\end{equation}
		it is easy to verify that 
		\begin{equation}
			\int \psi_n|\d f_n|^2\,\d\mm_n\to  \int \psi|\d f|^2\,\d\mm.
		\end{equation}
		By polarization, \eqref{eq: claim limite forte} follows. Assume now that $\tilde{\psi}_n\in \L^2(\X)$ satisfies $\tilde{\psi}_n\to \psi$ in $\L^2$ and that $\sup_n \Lip(g_n)<\infty$.
		We can estimate 
		\begin{equation}
			\left| \int\psi_n\langle \d f_n, \d g_n\rangle\,\d\mm_n-\int \tilde{\psi}_n\langle \d f_n, \d g_n\rangle\,\d\mm_n\right|\leq \|\psi_n-\tilde{\psi}_n\|_{\L^2(\mm_n)}\Lip(g)\ch(f_n)^\frac{1}{2},
		\end{equation}
		so \eqref{eq: claim limite forte} holds even with this assumptions.
		
		Let $G$ be any weak limit of $|\d f_n|$ in $\L^2$. By the convergence of the Cheeger energy and Lemma \ref{lemma: mima scs}, it is sufficient to prove that $\int \psi \langle \d f, \omega\rangle \,\d\mm\leq \int \psi G\,\d\mm$ for all $\omega\in\mathcal R$ and $\psi\in \L^2(\X)$, $\psi\geq 0$.
		Let $\omega=\sum_k \mathbbm{1}_{A_k}\d g_k\in \mathcal R$ and $\psi$ as above.
		Let $\varphi_{n, k}\in \L^2(\X_n)$, be such that $\varphi_{n, k}\to \mathbbm{1}_{A_k}$ in $\L^2$, $|\varphi_{n, k}|\leq 1$ $\mm_n$--a.e. and $\varphi_{n, k}\geq 0$; let $\psi_n \in \L^2(\X_n)$ be such $\psi_n\to \psi$ in $\L^2$ and $\psi_n\geq 0$; finally, let $g_{n, k}\in \Lip(\X_n)$ be as in Lemma \ref{lemma: limsup lipschitz}.
		We have
		\begin{align}
			\int \psi\langle \d f, \omega\rangle\,\d\mm&=\lim_n\sum_k\int\psi_n\varphi_{n, k}\langle \d f_n, \d g_n\rangle\,\d\mm_n\\
			&\leq \lims_n \int \psi_n (\sum_k \varphi_{n, k})|\d f_n|\,\d\mm_n\\
			&\leq \int\psi G\,\d\mm.
		\end{align}
	\end{proof}
	Once that the convergence of the pointwise norm of the differentials is assured, the machinery introduced in \cite{Gi23} allows to speak about convergence of vectors and tensors. The basic ideas is to define a sufficiently wide family of test sequence, which we implicitly assume to strongly converge and then to use it, as the name suggests, to test the convergence of the general sequence in a ``duality'' fashion. 
	\begin{deff}[Test sequences]
		Let $f_n\in \operatorname{Test}(\X_n)$, $f_\infty\in \operatorname{Test}(\X)$. We say that $(f_n)$ is a test sequence if
		\begin{align}
			\sup \|f_n\|_{\L^\infty}+\operatorname{Lip}(f_n)&+\|\Delta f_n\|_{\L^\infty}<+\infty\\
			f_n\to f_\infty\text{ in }\L^2\quad &\text{ and }\quad \ch(f_n)\to \ch(f)\\
			 \Delta f_n\to \Delta f_\infty\text{ in }\L^2\quad &\text{ and }\quad \sup \ch(\Delta f_n)<+\infty.
		\end{align}
		A test sequence of vectors is a sequence of the form $v_n=\sum_j f_{j, n} \nabla g_{j, n}$ with $(f_{j, n})_n,(g_{j, n})_n$ being test sequence. Finally, for $r>1$, a test sequence of $r$--tensors is a sequence of the form $ z_n=\sum_{j=1}^m f_{j, n}\nabla g_{1, j, n}\otimes\ldots\otimes\nabla g_{r, j, n}$, where $(f_{j, n})_n, (g_{i, j, n})_n$ are test sequences.
	\end{deff}
	\begin{deff}\label{def: conv vettori}
		Let $v_n\in \L^2(T^{\otimes r}\X_n)$ for $n\geq 1$ be a sequence of $r$-tensors. We say that $v_n\rightharpoonup v\in \L^2(T^{\otimes r}\X)$ in $\L^2$ if $\sup_n\|v_n\|_{\L^2(T^{\otimes r}\X_n)}<\infty$ and $\langle v_n, z_n\rangle\rightharpoonup \langle v, z\rangle$ in $\L^2$ for any test sequence of $r$--tensors $(z_n)$.
		We say that $v_n\to v$ in $\L^2$ if the above holds and moreover $\|v_n\|_{\L^2(T^{\otimes r}\X_n)}\to \|v\|_{\L^2(T^{\otimes r}\X)}$.
	\end{deff}
	\begin{rmk}
		Notice that if $f_n\in \L^2(\X_n)$ is a sequence of functions, the convergences now defined for $0$--tensors are equivalent to weak and strong $\L^2$ convergences as defined in Section \ref{section: conv}. This will follow by Remark \ref{oss: densità}. 
	\end{rmk}
	We list a few properties of test sequences, among which the fact that test sequences do converge strongly in the sense of Definition \ref{def: conv vettori}.
	\begin{lm}\label{lemma: convergenza prodotto}
		Let $(f_n)_{n\geq 1}, (g_n)_{n\geq 1}$ be test sequence. Then $f_n\d g_n\to f_\infty \d g_{\infty}$ in $\L^2$.
	\end{lm}
	\begin{proof}
		The statement follows immediately from Lemma \ref{lemma: convergenza moduli diff} and the arguments in the proof of \eqref{eq: claim limite forte}.
	\end{proof}
	\begin{cor}
		Let $(f_n)_{n\geq 1}, (g_n)_{n\geq 1}$ be test sequences of functions, $(v_n)_{n\geq 1}$ be a test sequence of vectors and $(z_n)_{n\geq 1}$ be a test sequence of tensors.
		Then the following hold:
		\begin{itemize}
			\item $(f_ng_n)_{n\geq 1}$ is a test sequence;
			\item $\d f_n\to \d f_{\infty}$ in $\L^2$;
			\item $|z_n|\to |z_\infty|$ in $\L^2$;
			\item $z_n\to z_\infty$ in $\L^2$;
			\item $\div(v_n)\to \div(v_\infty)$ in $\L^2$.
		\end{itemize}
	\end{cor}
	\begin{proof}
		The only non trivial parts in the proof of the fact that $f_ng_n$ is a test sequence are showing the convergence of the Cheeger energy and of the Laplacians, both of which follows from Lemma \ref{lemma: convergenza prodotto}. The convergence of the differential is again an immediate consequence of Lemma \ref{lemma: convergenza prodotto}, as is the convergence of the pointwise norms of test sequences of tensors and the strong convergence. Finally, Lemma \ref{lemma: convergenza prodotto} and the convergence of Laplacians of test sequences imply the convergence of the divergences of test sequences of vectors.
	\end{proof}
	\begin{rmk}\label{oss: densità}
		Notice that by a simple mollification argument it is easy to prove that $\{v_\infty: (v_n)_{n\geq 1} \text{ is test}\}$ is dense in $\L^2(T^{\otimes r}\X)$.
	\end{rmk}
	By the exact same proofs in \cite{Gi23}, we obtain the following list of properties for converging sequences:
	\begin{prop}\label{prop: proprietà convergenze}
		Let $(f_n)_{n\geq 1}$ be a sequence of functions with $f_n\in \L^2(\X_n)$ and $(v_n)_{n\geq 1}, (w_n)_{n\geq 1}$ be sequences of tensors, with $v_n, w_n\in \L^2(T^{\otimes r}\X_n)$. Then the following properties hold:
		\begin{enumerate}
			\item \label{it: unicità} Weak (and thus strong) limits are unique.
			\item \label{it: semic} $v_n\rightharpoonup v$ implies $|v|\leq G$ for any $\L^2$ weak limit $G$ of $|v_n|$.
			\item \label{it: conv} $v_n\to v$ implies $|v_n|\to |v|$ in $\L^2$.
			\item \label{it: coupling} If $v_n\rightharpoonup v$ and $w_n\to w$ then $\langle v_n, w_n\rangle$ converges weakly in $\L^1$ to $\langle v, w\rangle$.
			\item \label{it: lin} If $v_n\to v$ and $w_n\to w$ then $v_n+w_n\to v+w$.
			\item \label{it: compattezza debole} If $\sup_n \|v_n\|_{\L^2}<\infty$ then up to a subsequence $v_n$ is weakly converging.
			\item \label{it: chiusura differenziale} If $f_n\rightharpoonup f$ and $\sup \ch(f_n)<\infty$ then $df_n\rightharpoonup df$. If moreover $\ch(f_n)\to \ch(f)$, then $\d f_n\to \d f$.
			\item \label{it: ch hess} If $f_n\to f$, $\ch(f_n)\to \ch(f)$ and $\sup_n \|\operatorname{Hess} f_n\|_{\L^2}<\infty$, then $\operatorname{Hess} f_n\rightharpoonup \operatorname{Hess} f$.
			\item \label{it: ch derivata cov} If $v_n\to v$ and $\sup_n \|\nabla v_n\|_{\L^2}<\infty$, then $v\in W^{1, 2}_C(T\X)$ and $\nabla v_n\rightharpoonup \nabla v$.
		\end{enumerate}
	\end{prop}
	As a final application of this definitions and results, we prove the stability of solutions to the continuity equation and of flows of suitably converging vector fields. In order to do that, let us recall some basic definitions and results.
	\begin{deff}
		Let $(\mu_t)_{t\geq 0}\subset \pr_{ac}(\X)$ be a weakly continuous curve and let $(v_t)_{t>0}\in \L^\infty([0, 1], \L^2_x(T\X))$. We say that $\mu_t$ is a weak solution to the continuity equation $\partial_t\mu_t+\div(v_t\mu_t)=0$ with initial datum $\mu_0$ if for all $f\in \operatorname{TEST}(\X)$ the function $t\mapsto \int f\,d\mu_t$ is absolutely continuous and 
		\begin{equation}
			\frac{\d}{\d t} \int f\,d\mu_t=\int \d f(v_t)\,\d\mu_t
		\end{equation}
		for a.e. $t>0$.
	\end{deff}
	\begin{deff}[\cite{Gi23}, \cite{AmTr14}]
		Let $(v_t)_{t\geq 0}\in \L^2([0, 1]; \L^2(\X))$. A Borel map $F:[0, 1]\times \X\to \X$ is a Regular Lagrangian Flow for $(v_t)$ if the following hold:
		\begin{itemize}
			\item[i)] $t\mapsto F_t(x)$ is continuous for $\mm$--a.e. $x\in \X$;
			\item[ii)] The maps $F_t$ have equi--bounded compression;
			\item[iii)] For easch $f\in \W^{1, 2}$ the following holds: for $\mm$--a.e. $x\in \X$ the curve $t\mapsto f(F_t(x))$ is in $\W^{1, 1}$ with derivative $\d f(v_t)(F_t(x))$.
		\end{itemize}
	\end{deff}

	The following existence Theorem holds:
	\begin{thm}[\cite{Gi23}, \cite{AmTr14}]\label{teo: equazioni continuità}
		Let $\mu\in \pr (\X)$ be such that $\|\frac{\d\mu}{\d\mm}\|\in \L^\infty$ and let $v_t\in \L^\infty([0, 1]; \L^2_x(T\X))$ with
		\begin{equation}
			\mathbf N(v_\cdot)\coloneqq \int_0^1 \left[\||v_t|\|_{\L^\infty(\X)}+\|\div(v_t)\|_{\L^\infty(\X)}+\|\nabla v_t\|_{\L^2(T^{\otimes 2}\X)}\right]\,\d t<\infty.
		\end{equation}
	 	Then there exist a unique weak solution to the continuity equation
		\begin{equation}
			\partial_t \mu_t+\div(v_t\mu_t)=0
		\end{equation}
		with $\mu_0=\mu$. Moreover it holds
		\begin{equation}
			\|\rho_t\|_{\L^\infty(\X)}\leq \|\rho_0\|_{\L^\infty(\X)}e^{\int_0^t\|\div(v_t)^-\|_{\L^\infty(\X)}\,\d r},
		\end{equation} 
		for all $t>0$, where $\mu_t=\rho_t m$.
	\end{thm}
	\begin{thm}[\cite{Gi23}, \cite{AmTr14}]\label{teo: esistenza e unicità flussi}
		Let $\X$ be $\RCD(K, \infty)$ and $(v_t)$ be such that $\mathbf N(v_\cdot)<\infty$. Then $(v_t)$ admits a unique Regular Lagrangian Flow $F$. Moreover, given any probability measure $\mu$ such that $\mu\leq C\mm$ for some $C>1$, it holds 
		\begin{equation}
			(F_t)_\# \mu=\mu_t,
		\end{equation}
		where $\mu_t$ is as in Theorem \ref{teo: equazioni continuità}.
	\end{thm}
	We are now ready to define the right notion of convergence and prove this final stability result.
	\begin{deff}
		Let $(v^n_t)_t\in \L^2([0, 1]; \L^2_x(T\X_n))$ for $n\geq 1$. We say that $v^n\to v^\infty$ strongly in time, strongly in space if
		\begin{equation}
			\lim_n \iint \varphi_t \langle v^n_t, z_n\rangle\,\d\mm_n \d t=\iint \varphi_t \langle v^\infty_t, z_\infty\rangle\,\d\mm\d t 
		\end{equation}
		for all $\varphi \in C^0_\b([0, 1])$ and $(z_n)_{n\geq 1}$ test sequence of vectors and moreover  
		\begin{equation}
			\iint|v^n_t|^2\,\d\mm_n\d t\to \iint|v^\infty_t|^2\,\d\mm\d t.
		\end{equation}
	\end{deff}
	\begin{rmk}\label{oss: conv norma}
		Notice that whenever we have a sequence $\X_n$ with projections $p_n$ as before, we can produce a sequence of projections $p_n'\coloneqq (p_n, id):\X_n\times [0, 1]\to \X\times [0, 1]$, where $[0, 1]$ is endowed with the Lebesgue measure. This can be useful (and will be used in the following) to study curves of functions.
		By reasoning as in the proof of \ref{it: semic} and \ref{it: conv} in Proposition \ref{prop: proprietà convergenze}, it is easy to see that if $v_n\to v$ strongly in time, strongly in space, then $|v_n|\to |v|$ in $\L^2_{t, x}$, i.e. in $\L^2$ with respect to the spaces $\X_n\times (0, 1)$ projecting on $\X\times (0, 1)$.
	\end{rmk}
	\begin{thm}\label{teo: stab eq cont}
		Let $\mu_n\in\pr (\X_n)$ be such that $\sup_n \|\frac{\d\mu_n}{\d\mm_n}\|_\infty<+\infty$ and $\mu_n\rightharpoonup\mu$. Let $v^n_t\in \L^\infty([0, 1]; \L^2(T\X_n))$ converge strongly in time, strongly in space to $v_t$ with $\sup_n \mathbf N(v_n)<+\infty$.
		Then the solution $\mu_{n, t}$ of the continuity equations $\partial_t\mu_{n, t}+\div(v^n_t\mu_{n, t})=0$ with initial datum $\mu_n$ for every $t\in [0, 1]$ converge weakly to $\mu_t$, the solution to $\partial_t\mu_t+\div(v_t\mu_{t})=0$ starting from $\mu$.
	\end{thm}
	\begin{proof}
		Let us start by noticing that by Theorem \ref{teo: equazioni continuità} we have that $\mu_{n, t}\leq C\mm_n$ for some $C>0$ and all $n, t$. Moreover, by the uniform bound on $\||v_n|\|_{\infty}$ and the superposition principle in \cite[Proposition 5.17]{Gi23} the $\W_2$ speeds of $\mu_{n, t}$ are equibounded. By an Ascoli--Arzelà argument similar to Theorem \ref{teo: conv heat}, up to a subsequence we have $\mu_{n, t}\rightharpoonup\nu_t$ for all $t\in[0, 1]$ for some weakly continuous curve $\nu$. Notice that this convergence implies that $\int \mu_{n, t}\,\d t\rightharpoonup \int \nu_t\,\d t$ in the sense of Section \ref{section: conv}. By the equiboundedness then $\frac{\d\mu_{n, \cdot}}{\d\mm_n}\rightharpoonup \frac{\d\nu_\cdot}{\d\mm}$ in $\L^p$ for all $p<\infty$. Let now $f_n$ be any test sequence. Reasoning as in \cite[Theorem 6.12]{Gi23}, it suffices to prove that
		\begin{equation}\label{eq: conv derivate}
			\int \d f_n(v_t^n)\,\d\mu_{n, t}\rightharpoonup\int \d f_\infty(v^n)\,d\nu_t
		\end{equation}
		in $\L^2(0, 1)$.
		By Remark \ref{oss: conv norma} and polarization we see that $\d f_n(v^n_\cdot)\to \d f_\infty(v^n)$ in $\L^p_{t, x}$ for all $p<\infty$. It is now easy to see that \eqref{eq: conv derivate} holds. 
	\end{proof}
	\begin{rmk}
		The hypothesis on the fields $v^n$ could be weakened to a suitable weak in time, strong in space convergence by reasoning as in \cite{Gi23}.
	\end{rmk}
	\begin{thm}
		Let $(v_n)_{n\geq 1}, v$ as in Theorem \ref{teo: stab eq cont}, let $F_n:\X_n\to C^0([0, 1];\X_n)$ be the flow of $v_n$ and $F_\infty:\X\to \ C^0([0, 1]; \X)$ the flow of $v$. Then 
		\begin{equation}
			F_n(\cdot, t)\to F_\infty(\cdot, t)
		\end{equation} 
		in $\L^0$ for all $t\in [0, 1]$.
	\end{thm}
	\begin{proof}
		The result follows immediately from Theorem \ref{teo: stab eq cont}, Theorem \ref{teo: esistenza e unicità flussi} and Proposition \ref{prop: L0 via misure bis}.
	\end{proof}

%	\bibliographystyle{plain} % We choose the "plain" reference style
%	\bibliography{Bibliografia\\bibliografia} % Entries are in the refs.bib file

\end{document}